\documentclass{amsart}
\usepackage{euscript,amsmath, amssymb, amsfonts, mathtools}
\usepackage{mathrsfs}
\usepackage{mathabx}
\usepackage{stmaryrd}
\usepackage{crossreftools}

\makeatletter
\@namedef{subjclassname@1991}{Subject}
\makeatother

\usepackage[hypertexnames = false, colorlinks=true, allcolors=blue]{hyperref}

\usepackage{tikz-cd}

\numberwithin{equation}{section}
\usepackage{enumitem}   

\newtheorem{theorem}{Theorem}[section]
\newtheorem{lemma}[theorem]{Lemma}
\newtheorem{proposition}[theorem]{Proposition}
\newtheorem{corollary}[theorem]{Corollary}

\theoremstyle{definition}
\newtheorem{definition}[theorem]{Definition}

\newtheorem{algo}{Algorithm}[section]

\theoremstyle{remark}
\newtheorem*{remark}{Remark}
\newtheorem*{remarks}{Remarks}
\usepackage{graphicx}  
\usepackage{icomma}
\usepackage{bbm}
\usepackage{tikz}
\usetikzlibrary{shapes}
\usepackage{color}
\usepackage{verbatim}
\usepackage{xurl}
\usepackage{indentfirst}
\usepackage{lipsum}
\usepackage[backend = biber]{biblatex}
\usepackage{placeins}
\usepackage{bigints}
\usepackage{subcaption}
\usepackage{appendix}
\DeclareMathOperator{\Span}{Span}
\DeclareMathOperator{\Jac}{Jac}

\DeclareMathOperator{\CC}{\mathbb C}
\DeclareMathOperator{\RR}{\mathbb R}
\DeclareMathOperator{\ZZ}{\mathbb Z}
\DeclareMathOperator{\NN}{\mathbb N}
\DeclareMathOperator{\PP}{\mathbb P}

\DeclareMathOperator{\Per}{Per}

\DeclareMathOperator{\Res}{\mathcal R}
\DeclareMathOperator{\Discr}{\mathcal D}

\usepackage{csquotes}

\begin{document}
	\title[Computation of genus 2 Kleinian hyperelliptic functions]{Computation of genus 2 Kleinian hyperelliptic functions via Richelot isogenies}
	
	\author{Matvey Smirnov}
	\address{119991 Russia, Moscow GSP-1, ul. Gubkina 8,
		Institute for Numerical Mathematics,
		Russian Academy of Sciences}
	\email{matsmir98@gmail.com}

	\begin{abstract}
		In this work we propose an algorithm that numerically evaluates Kleinian hyperelliptic functions associated with a complex curve of genus 2. This algorithm is based upon constructing a sequence of curves with Richelot isogenous Jacobians and a recurrent procedures that reduces the calculation to a degenerate curve. As a part of mentioned algorithm we propose a method of choosing a Richelot isogenous curve (among 15 possibilities) that guarantees convergence of the equations of the curves and associated Kleinian functions of weight 2 under iterations.
		
	\end{abstract}
	\maketitle
	
	\section{Introduction}
	This work is the final part of the cycle of papers that are devoted to giving an algorithm that computes Kleinian hyperelliptic functions associated with complex curves of genus 2. We refer to the previous parts~\cite{KleinianWeight2} and~\cite{RichelotExpression} and references therein for context and motivation.
	
	At first let us recall some of the results of~\cite{KleinianWeight2} and~\cite{RichelotExpression} (the precise facts and formulas will be stated in Section~\ref{sec:Prelim}). In~\cite{KleinianWeight2} we have associated a quadruple of special functions $\mathscr S_f = (S^f, S_{22}^f, S_{12}^f, S_{11}^f)$ (which we called Kleinian functions of weight 2) with any complex polynomial of degree $5$ or $6$ without multiple roots. These functions are closely related with the corresponding to $f$ curve of genus $2$ (namely, the curve given by the equation $y^2 = f(x)$). In particular these functions can be used to embed the Kummer surface of the mentioned curve into $\CC\PP(3)$. Moreover, in~\cite{KleinianWeight2} we presented the explicit formulas that express classical Kleinian hyperelliptic functions $\wp_{jk}^f$, $\zeta_{j}^f$, and $\sigma^f$ through $\mathscr S_f$. The paper~\cite{RichelotExpression} deals with the relation between $\mathscr S_f$ and $\mathscr S_{\hat f}$, where $\hat f$ is the equation of the curve, which has Richelot isogenous Jacobi variety with the Jacobian variety of the initial curve. In fact, we found an explicit expression of $\mathscr S_f$ through $\mathscr S_{\hat f}$.
	
	Now, using the foregoing results we return to the idea of the numerical algorithm. That is, given $f$ to compute $\mathscr S_f$ at first we find a sequence of polynomials $f^{(0)} = f$, $f^{(1)}$, $\dots$, $f^{(n)}$, where each $f^{(k)}$ is obtained via Richelot's construction from $f^{(k-1)}$. If the functions $\mathscr S_{f^{(n)}}$ can be somehow approximated, then, by using the results of~\cite{RichelotExpression}, we can calculate recursively $\mathscr S_{f^{k-1}}$ through $\mathscr S^{f^{(k)}}$ and, as the result, obtain an approximation of $\mathscr S_f$. The remaining difficulty is the requirement that $\mathscr S_{f^{(n)}}$ can be effectively approximated. Even for the similar algorithms in genus $1$ (i.e. Landen's or AGM methods; see, e.g.~\cite{Cremona} and~\cite{Smirnov}) it is known that the effectiveness (and even convergence itself) heavily depends on the choice of the next isogenous curve on each step. For the Richelot's construction there are 15 distinct choices of the next curve (in the generic case), corresponding to the partitions of the roots of $f$ (counting $\infty$ if $\deg f = 5$) into three pairs. Thus, the missing part of the desired algorithm is the method to choose such partition on each step, such that the resulting sequence of polynomials $\{f^{(n)}\}$ converges and yields convergent sequence of functions $\{\mathscr S_{f^{(n)}}\}$ with explicitly computable limit.
	
	In this work we propose the following idea. Let $D_1, D_2, D_3 \subset \CC\PP(3)$ be disjoint open discs. We say that $f$ is subordinate to the triple of discs $(D_1,D_2,D_3)$, if each of these discs contains exactly two roots of $f$. Clearly, the described situation naturally splits the roots of $f$ into three pairs. We prove that the Richelot's construction applied to $f$ (with respect to the foregoing partition into three pairs) yields a polynomials that again does not have multiple roots and is subordinate to $(D_1,D_2,D_3)$. Moreover, the iterations of the described procedure form a sequence of polynomials $f^{(n)}$ that quadratically converges to a polynomial with three double roots. Finally, we prove that the sequence $\{\mathscr S_{f^{(n)}}\}$ of associated Kleinian functions of weight 2 converges (uniformly on compact sets), and we find an expression for the limit. Based on these results we propose an algorithm to numerically compute Kleinian functions associated with curves of genus 2.
	
	The paper is organized as follows. In Section~\ref{sec:Prelim} we collect all necessary statements and formulas from~\cite{KleinianWeight2} and~\cite{RichelotExpression}. Section~\ref{sec:Iter} contains the analysis of the Richelot's construction (and its iterations) given that a polynomial is subordinate to a triple of disjoint discs. Then, in Section~\ref{sec:Limit}, we prove that the corresponding Kleinian functions of weight 2 converge and find the expression for the limit. Finally, in Section~\ref{sec:Algo} we formulate the algorithms.
	\section{Preliminaries}\label{sec:Prelim}
	
	\subsection{Curves of genus 2 and Kleinian functions}
	We shall denote the space of complex polynomials of degree $\le k$ by $\mathfrak P_k$. Given a polynomial $f$ we denote by $f_j$ the coefficient of $x^j$ in $f$. Moreover, we call a polynomial $f \in \mathfrak P_6$ {\it{admissible}} if it does not have multiple roots and $\deg f$ is equal to either $5$, or $6$. A polynomial $f$ of degree $5$ with leading coefficient $f_5 = 4$ we call a polynomial in {\it{Weierstrass form}}. With an admissible polynomial we associate a genus 2 curve $\mathcal X_f$ (whose affine part is defined by the equation $y^2 = f(x)$). On this curve we fix the $1$-forms
	\[
	\omega_1^f = \frac{dx}{y},\;\;\omega_2^f = \frac{xdx}{y},\;\;r_1^f = \frac{f_3x + 2f_4x^2 + 3f_5x^3 + 4f_6x^4}{4y}dx,\;\;r_2^f = \frac{f_5x^2 + 2f_6x^3}{4y}dx.
	\]
	The forms $\omega_1^f, \omega_2^f$ constitute a basis of the space of holomorphic $1$-forms on $\mathcal X_f$. We denote the lattice of periods of forms $\omega_1^f, \omega_2^f$ by $\Per_f \subset \CC^2$. This lattice is canonically equipped with the intersection pairing $\langle \cdot, \cdot\rangle_f$, induced from $H_1(\mathcal X_f, \ZZ)$. Moreover, given a period $w \in \Per_f$ we define $\eta^f(w) \in \CC^2$ by the rule
	\begin{equation*}
		\eta^f(w) = -\int_\gamma \begin{pmatrix} r_1^f \\ r_2^f\end{pmatrix},\text{ where } w = \int_\gamma \begin{pmatrix} \omega_1^f \\ \omega_2^f \end{pmatrix}.
	\end{equation*}
	The periods $\eta^f$ are well-defined, since $\rho_1^f$ and $\rho_2^f$ are of the second kind.
	
	We recall~\cite{KleinianWeight2} the definition of the space $\mathfrak S_f$ of {\it{Kleinian hyperelliptic functions of weight 2}}. It consists of all holomorphic functions $\phi$ on $\mathbb C^2$ that satisfy the property
	\begin{equation}\label{eqSpaceDef}
		\phi(z + w) = \exp\left[2\eta^f(w)^T\left(z + \frac{w}{2}\right)\right]\phi(z)
	\end{equation}
	for all $z \in \CC^2$ and $w \in \Per_f$. It is proved (see~\cite[Proposition~3.1~(iv)]{KleinianWeight2}) that this space is four-dimensional and an element $\phi \in \mathfrak S_f$ is uniquely determined by its order 2 Taylor expansion at zero. Therefore, $\mathfrak S_f$ is spanned by the four functions $S^f$, $S^f_{11}$, $S^f_{12}$, $S^f_{22}$, where $S^f$ is the unique element of $\mathfrak S_f$ that satisfies $S^f(z) = z_1^2 + \bar{o}(z^2)$, and the other basis elements have the expansions
	\begin{equation}\label{eqTaylorExpansions}
		S_{22}^f(z) = 2z_1z_2 + \bar{o}(z^2),\;\;
		S_{12}^f(z) = -z_2^2 + \bar{o}(z^2),\;\;
		S_{11}^f(z) = 1 + \bar{o}(z^2).
	\end{equation}
	Note, that $S_{jk}^f = S^f\wp_{jk}^f$ for all $j,k$, where $\wp_{jk}^f$ are the Kleinian $\wp$-functions associated with the curve $\mathcal X_f$ (in fact, in the paper~\cite{KleinianWeight2} the functions $S_{jk}^f$ were {\it{defined}} using direct formulas involving $\wp$-functions, while the expansions~\eqref{eqTaylorExpansions} were derived from this definition, see~\cite[Theorem~4.1]{KleinianWeight2}; here we are going to use only the expansions~\eqref{eqTaylorExpansions}).
	
	Finally, we need the relation between Kleinian functions of weight 2 and theta functions of weight 2 (see~\cite[Definition~II.1.2]{mumfordI}). Given a {\it{Riemann matrix}} $\Omega \in \CC^{2 \times 2}$ (a symmetric matrix with positive definite imaginary part) we consider the space $R_2^\Omega$ that consists of all holomorphic functions $\phi$ on $\CC^2$ such that 
	\[
	\phi(z + n + \Omega m) = \exp\left[-2i\pi m^T \Omega m -4i\pi m^T z\right]\phi(z)
	\]
	for all $z \in \CC^2$ and all $m,n \in \ZZ^2$. Assume that $a_1,a_2,b_1,b_2 \in \Per_f$ is a {\it{symplectic basis}} (that is, $\langle a_j, b_j\rangle_f = 1$ for $j = 1,2$ and $\langle a_1, a_2\rangle_f = \langle b_1, b_2\rangle_f = \langle a_1, b_2\rangle_f = \langle a_2, b_1\rangle_f = 0$), and introduce the matrices
	\[
	A = \begin{pmatrix}
		a_1 & a_2
	\end{pmatrix},\;\;
	B = \begin{pmatrix}
		b_1 & b_2
	\end{pmatrix},\;\;
	\eta_A = \begin{pmatrix}
		\eta^f(a_1) & \eta^f(a_2)
	\end{pmatrix},\;\;
	\eta_B = \begin{pmatrix}
		\eta^f(b_1) & \eta^f(b_2)
	\end{pmatrix}.
	\]
	Then it is well-known that the matrix $\Omega = A^{-1}B$ is a Riemann matrix. Moreover, the transformation $T$ defined by the rule
	\begin{equation}\label{eqTisomorphism}
		T(\phi)(z) = \exp\left[z^T (\eta_AA^{-1})z\right]\phi(A^{-1}z)
	\end{equation}
	is an isomorphism of the space $R^\Omega_2$ onto $\mathfrak S_f$ (see~\cite[Proposition~3.1]{KleinianWeight2}).
	
	\subsection{Operations with polynomials of degree 2}
	
	Here we reproduce some notation from~\cite{RichelotExpression}. That is, for $p,q,r \in \mathfrak P_2$ we define
	\[
	[p,q] = p'q - pq',\;\;\Delta(p,q,r) = \begin{pmatrix}
		p_0 & q_0 & r_0 \\
		p_1 & q_1 & r_1 \\
		p_2 & q_2 & r_2
	\end{pmatrix},
	\]
	\begin{equation*}\label{eqDiscrAndResDef}
		\Discr(p) = p_1^2 - 4p_0 p_2,\;\; \Res(p,q) = (p_2 q_0 - p_0 q_2)^2 + (p_2 q_1 - p_1 q_2)(p_0 q_1 - p_1 q_0).
	\end{equation*}
	It is easy to see that $\Discr(p)$ and $\Res(p,q)$ coincide with discriminant of $p$ and resultant of $p$ and $q$ respectively, assuming that $\deg p = \deg q = 2$. That is, if $p(x) = a(x - t_1)(x - t_2)$ and $q(x) = b(x - s_1)(x - s_2)$, then  \begin{equation}\label{eqDiscrAndResRootFormula}
		\Discr(p) = a^2(t_1 - t_2)^2,\;\;\Res(p,q) = a^2b^2(t_1 - s_1)(t_1 - s_2)(t_2 - s_1)(t_2 - s_2).
	\end{equation}
	
	The necessary properties of the introduced constructions are listed in the following propositions (for the proofs see~\cite[Subsection~2.2]{RichelotExpression}).
	
	\begin{proposition}\label{propBasicBasic}
		The following statements hold for all non-zero $p,q \in \mathfrak P_2$.
		\begin{enumerate}[label=(\roman*)]
			\item\label{BasicBasici} $\Discr(p) = 0$ if and only if $p$ has only one root (of multiplicity two).
			\item\label{BasicBasicii} $\Res(p,q) = 0$ if and only if $p$ and $q$ have a common root.
			\item\label{BasicBasiciii} $[p,q] = 0$ if and only if $p$ and $q$ are proportional.
			\item\label{BasicBasiciv} $\Discr([p,q]) = 4\Res(p,q)$.
			\item\label{BasicBasicv} Let $S$ be a M\"obius transformation, i.e.
			\[
			S(x) = \frac{ax + b}{cx + d}, \;\;\det \begin{pmatrix} a & b \\ c & d \end{pmatrix} = 1.
			\]
			Then $[(cx + d)^2p \circ S, (cx + d)^2q \circ S] = (cx + d)^2[p,q] \circ S$.
		\end{enumerate}
	\end{proposition}
	
	\begin{proposition}\label{propBasicProperties}
		Assume that $p,q,r \in \mathfrak P_2$ and let $\hat p = [q,r]$, $\hat q = [r,p]$, and $\hat r = [p,q]$. Then the following statements hold.
		\begin{enumerate}[label=(\roman*)]
			\item\label{BasicPropertiesi} $[\hat p, \hat q] = -2\Delta(p,q,r)r$, $[\hat q, \hat r] = -2\Delta(p,q,r)p$, and $[\hat r, \hat p] = -2\Delta(p,q,r)q$.
			\item\label{BasicPropertiesii} $\Delta(\hat p, \hat q, \hat r) = -2\Delta(p,q,r)^2$. 
			\item\label{BasicPropertiesiii} $\Res(\hat p, \hat q) = \Delta(p,q,r)^2\Discr(r)$, $\Res(\hat p, \hat r) = \Delta(p,q,r)^2\Discr(q)$, and $\Res(\hat q, \hat r) = \Delta(p,q,r)^2\Discr(p)$.
			\item\label{BasicPropertiesiv} Let $S$ be a M\"obius transformation as in Proposition~\ref{propBasicBasic}~\ref{BasicBasicv}. Then
			\[
			\Delta( (cx + d)^2p \circ S, (cx + d)^2q \circ S, (cx + d)^2r \circ S) = \Delta(p,q,r).     
			\]
			\item\label{BasicPropertiesv} Assume that $\Delta(p,q,r) \ne 0$ and that $f = pqr$ is admissible. Then the polynomial $\hat p \hat q \hat r$ is also admissible.
		\end{enumerate}
	\end{proposition}
	
	\subsection{Richelot isogeny}
	
	Till the end of this section we fix $p,q,r \in \mathfrak P_2$ such that $f = pqr$ is admissible and $\Delta(p,q,r) \ne 0$. From Proposition~\ref{propBasicProperties}~\ref{BasicPropertiesv} follows that the polynomial
	\[
	\hat {f} = \frac{1}{4\Delta(p,q,r)}\hat p\hat q \hat r
	\]
	is admissible, where $\hat p = [q,r]$, $\hat q = [r, p]$, and $\hat r = [p,q]$. As in~\cite{RichelotExpression} we shall refer to $\hat f$ as the result of the {\it{Richelot construction associated with the decomposition}} $f = pqr$. It is known (see, e.g.~\cite{Humbert},~\cite{BostMestre}), that between the curves $\mathcal X_f$ and $\mathcal X_{\hat f}$ there is a canonical correspondence, that induces a $(2,2)$-isogeny between the corresponding Jacobi varieties. We shall use the following statements about the relationship between $\mathcal X_f$ and $\mathcal X_{\hat f}$ (see~\cite[Subsection~2.2]{RichelotExpression}).
	
	\begin{proposition}\label{propRichProperties}
		Let $f$ and $\hat{f}$ be as defined above. Then the following statements hold.
		\begin{enumerate}[label=(\roman*)]
			\item\label{Richi} $2\Per_f \subset \Per_{\hat f} \subset \Per_{f}$ and for all $w_1,w_2 \in \Per_{\hat f}$ the equality $\langle w_1, w_2 \rangle_{f} = 2\langle w_1, w_2\rangle_{\hat f}$ holds.
			\item\label{Richii} Assume that $a_1,a_2,b_1,b_2 \in \Per_f$ is a symplectic basis such that $a_1,a_2\in \Per_{\hat f}$. Then $a_1,a_2,2b_1,2b_2$ is a symplectic basis in $\Per_{\hat f}$. Moreover, such basis always exists.
		\end{enumerate}
	\end{proposition}
	
	Finally, we state the relationships between Kleinian functions of weight 2 associated with $f$ and $\hat f$, that was obtained in~\cite{RichelotExpression}. In order to do this we need some more notation. At first we introduce
	\begin{equation*}
		\mu_{jklm}(p,q,r) = \hat p_j p_k p_lq_mr_m + \hat q_j q_k q_lp_mr_m + \hat r_j r_k r_lp_mq_m,
	\end{equation*}
	where indices $j,k,l,m$ vary over the set $\{0,1,2\}$. Moreover, let \[
	\begin{split}
		\psi_0(p,q,r) = & \; 4\mu_{0002}(p,q,r) + \mu_{2110}(p,q,r) + \mu_{2001}(p,q,r), \\
		\psi_2(p,q,r) = & -4\mu_{2220}(p,q,r) - \mu_{0112}(p,q,r) - \mu_{0221}(p,q,r).
	\end{split}  
	\]
	Finally, let the matrix $\mathfrak H(p,q,r) \in \CC^{2 \times 2}$ be defined by the equality
	\begin{multline}\label{eqChiFormula}
		8\mathfrak H^{p,q,r} = \\ \begin{pmatrix}
			p_0q_1r_1 + p_1q_0r_1 + p_1q_1r_0 + \dfrac{\psi_0(p,q,r)}{\Delta(p,q,r)} & -\dfrac{\mu_{1021}(p,q,r)}{\Delta(p,q,r)} \\
			-\dfrac{\mu_{1021}(p,q,r)}{\Delta(p,q,r)} & p_2q_1r_1 + p_1q_2r_1 + p_1q_1r_2 - \dfrac{\psi_2(p,q,r)}{\Delta(p,q,r)}
		\end{pmatrix}
	\end{multline}
	
	With the foregoing definitions we now can formulate the following result (which is a combination of \cite[Lemma~3.3]{RichelotExpression} with the expression~\eqref{eqChiFormula} for the matrix $\mathfrak H(p,q,r)$ given in \cite[Subsection~4.2]{RichelotExpression}).
	
	\begin{proposition}\label{propEtaTransformation}
		For all $w \in \Per_{\hat f}$ the equality
		\begin{equation}\label{eqEtaTransformation}
			\eta^f(w) = 2\eta^{\hat f}(w) + \mathfrak H^{p,q,r}w
		\end{equation}
		holds.
	\end{proposition}

	Finally, we introduce the holomorphic mapping $\mathscr S_f:\CC^2 \to \CC^4$ defined by the formula
	\begin{equation*}
		\mathscr S_f(z) = \begin{pmatrix}
			S^f(z) & S_{22}^f(z) & S_{12}^f(z) &  S_{11}^f(z)
		\end{pmatrix}^T.
	\end{equation*}
	Now we state the main result of the paper~\cite{RichelotExpression}, which gives a formula to express $\mathscr S_f(z)$ through $\mathscr S_{\hat f}(z)$.

	\begin{theorem}\label{thMainTransformation}
		Let $\mathfrak A_{p,q,r}$, $\mathfrak A_{p,q,r}^{(22)}$, $\mathfrak A_{p,q,r}^{(12)}$, and $\mathfrak A_{p,q,r}^{(11)}$ be the $4 \times 4$ matrices defined in~\cite[Definition~4.5]{RichelotExpression}. Then the equality
		\begin{equation}\label{eqSMainTransform}
			\mathscr S_f(z) = -\frac{\exp\left(z^T\mathfrak H^{p,q,r}z\right)}{32\Delta(p,q,r)^3}\begin{pmatrix}
				\left(\mathscr S_{\hat f}(z)\right)^T \mathfrak A_{p,q,r}\mathscr S_{\hat f}(z) \\
				\left(\mathscr S_{\hat f}(z)\right)^T \mathfrak A_{p,q,r}^{(22)}\mathscr S_{\hat f}(z) \\
				\left(\mathscr S_{\hat f}(z)\right)^T \mathfrak A_{p,q,r}^{(12)}\mathscr S_{\hat f}(z) \\
				\left(\mathscr S_{\hat f}(z)\right)^T \mathfrak A^{(11)}_{p,q,r}\mathscr S_{\hat f}(z)
			\end{pmatrix}
		\end{equation} holds for all $z \in \CC^2$.
	\end{theorem}
	\begin{remark}
		We do not give here an explicit formula for the matrices $\mathfrak A_{p,q,r}$ and $\mathfrak A_{p,q,r}^{(jk)}$, since it is very cumbersome and is useful only for the computer realisation of the algorithms. If needed, this formula can be easily derived from \cite[Definition~4.5]{RichelotExpression}.
	\end{remark}

	\section{Iterations of the Richelot construction}\label{sec:Iter}
	
	Due to the nature of the isogeny based numerical methods, we need to iterate the Richelot construction to achieve a curve, for which the associated special functions can be effectively approximated. Therefore, on each step of the algorithm we need to specify the choice of a decomposition $f = pqr$, or, in other words, the partition of the roots of $f$ into three pairs. In this section we propose such specification.
	
	\subsection{Polynomials subordinate to a triple of disks}
	\begin{definition}
		A set $D \subset \CC\PP(1)$ is called an open disc (resp. a closed disc) if it is an image of the open unit disc (resp. the closed unit disc) in $\CC$ under a suitable M\"obius transformation.
	\end{definition}
	
	\begin{definition}
		Let $\mathfrak D = (D_1, D_2, D_3)$ be a triple of disjoint open disks in $\CC\PP(1)$. We say that non-zero $f \in \mathfrak P_6$ is subordinate to the triple $\mathfrak D$ if each of the disks $D_j$ contains exactly two roots of $f$ counting multiplicity (we consider infinity to be the root of $f$ of multiplicity $6 - \deg f$). 
	\end{definition}

	\begin{proposition}\label{propDiskSubordination}
		Let $\mathfrak D = (D_1, D_2, D_3)$ be a triple of disjoint open disks and let $p_1,p_2,p_3 \in \mathfrak P_2$ be non-zero polynomials such that both roots of $p_j$ are contained in $D_j$, $j = 1,2,3$. Then $\Delta(p_1, p_2, p_3) \ne 0$ and the polynomial $[p_2,p_3][p_3, p_1][p_1,p_2]$ is subordinate to $\mathfrak D$. 
	\end{proposition}
	
	The proof of Proposition~\ref{propDiskSubordination} is based on the following simple lemma.
	
	\begin{lemma}\label{lemDiskSubordination}
		Let $D$ be an open disc and let $D' = \CC\PP(1) \setminus \overline{D}$. Assume that $p,q \in \mathfrak P_2$ and that both roots of $p$ belong to $D$ and both roots of $q$ belong to $D'$. Then the polynomial $[p,q]$ has one root in $D$ and one root in $D'$.
	\end{lemma}
	\begin{proof}
		Due to projective invariance (i.e. Proposition~\ref{propBasicBasic}~\ref{BasicBasicv}) we can assume without loss of generality that $D = \{z \in \CC: \mathrm{Re}\;z > 0\}$, so $D' = \{z \in \CC: \mathrm{Re}\;z < 0\}$. Moreover, we can assume that $p$ and $q$ are monic, so $p(x) = (x - e_1)(x - e_2)$ and $q(x) = (x - h_1)(x - h_2)$, where $e_1, e_2 \in D$ and $h_1,h_2 \in D'$. The direct calculation now shows that
		\[
		[p,q](x) = (e_1 + e_2 - h_1 - h_2)x^2 + 2(h_1h_2 - e_1e_2)x + e_1e_2(h_1 + h_2) - h_1h_2(e_1 + e_2).
		\]
		It is straigtforward to see that the leading coefficient of $[p,q]$ cannot be equal to $0$, so infinity cannot be a root of $[p,q]$. Assume for a moment that there exist $e_1,e_2 \in D$ and $h_1,h_2 \in D$ such that $[p,q]$ has a purely imaginary root $ir$, $r \in \RR$. Consider the M\"obius transformation $S(z) = (z - ir)^{-1}$. It is clear that $S(D) = D$ and $S(D') = D'$. Further, consider polynomials $\tilde p = (z - ir)^2p\circ S$ and $\tilde q = (z - ir)^2q \circ S$. By Proposition~\ref{propBasicBasic}~\ref{BasicBasiciv} the polynomial $[p,q]$ has a root at infinity, while both roots of $p$ are in $D$ and both roots of $q$ are in $D'$. We have arrived at a contradiction after assuming that $[p,q]$ has a purely imaginary root. Thus, for all $e_1,e_2 \in D$ and $h_1,h_2 \in D'$ the polynomial $[p,q]$ does not have a purely imaginary root. Thus either $[p,q]$ has both roots in $D$, or it has both roots in $D'$, or it has one root in $D$ and one root in $D'$. Due to continuity of roots one of listed possibilities holds for all polynomials $p$ and $q$, so it boils down to considering a concrete example. Let $p(x) = (x - 1)^2$ and $q(x) = (x + 1)^2$. Then $[p,q] = 4(x -1)(x + 1)$. Thus, one of the roots of $[p,q]$ belongs to $D$ and the other one to $D'$.
	\end{proof}
	
	\begin{proof}[Proof of Proposition~\ref{propDiskSubordination}.]
		From Lemma~\ref{lemDiskSubordination} it is clear that $D_1$ contains one of the roots of polynomials $[p_1,p_2]$ and $[p_3,p_1]$ and does not contain any roots of $[p_2,p_3]$. Clearly, similar statement holds for $D_2, D_3$, so the polynomial $[p_2,p_3][p_3, p_1][p_1,p_2]$ is subordinate to disks $D_1, D_2, D_3$. Now we prove that $\Delta(p_1,p_2,p_3) \ne 0$. If $\Delta(p_1,p_2,p_3) = 0$, then from Proposition~\ref{propBasicProperties}~\ref{BasicPropertiesi} it follows that $[[p_1,p_2], [p_2,p_3]] = 0$. From Proposition~\ref{propBasicBasic}~\ref{BasicBasiciii} $[p_1,p_2]$ and $[p_2,p_3]$ are proportional, so they have the same roots. But this clearly is not true, as $[p_1,p_2]$ has a root in $D_1$ and $[p_2,p_3]$ does not.
	\end{proof}
	
	Assume that a polynomial $f \in \mathfrak P_6$ is subordinate to a triple $\mathfrak D = (D_1,D_2,D_3)$. Then $f$ can be written in the form $f = p_1p_2p_3$, where $p_j \in \mathfrak P_2$ and both roots of $p_j$ are contained in $D_j$ for all $j \in \{1,2,3\}$. We define $H_{\mathfrak D}(f)$ to be the polynomial
	\[
	H_{\mathfrak D}(f) = \frac{1}{4\Delta(p_1,p_2,p_3)}[p_2, p_3][p_3, p_1][p_1, p_2].
	\]
	It is clear that the triple $(p_1,p_2,p_3)$ is determined uniquely up to the transformations $(p_1,p_2,p_3) \to (\alpha p_1, \beta p_2, \gamma p_3)$, where $\alpha, \beta, \gamma \in \CC$ and $\alpha\beta\gamma = 1$. Thus, $H_{\mathfrak D}(f)$ indeed depends only on $f$ and $\mathfrak D$. From Proposition~\ref{propDiskSubordination} it follows that $H_{\mathfrak D}(f)$ is again subordinate to the discs $\mathfrak D$ and it is an admissible polynomial if $f$ is admissible.
	
	Now let $\gamma_j$ be the circle enclosing $D_j$ (i.e. the boundary of $D_j$ oriented in the standard way). Since there exists a single-valued (meromorphic) branch of the square root of $f$ defined in a neighbourhood of the set $\CC\PP(1) \setminus (D_1 \cup D_2 \cup D_3)$ we can consider $2 \times 3$ matrices $\mathfrak W(f;\mathfrak D)$ and $\mathfrak E(f;\mathfrak D)$ defined by the equalities
	\begin{equation}\label{eqWMatrix}
		\mathfrak W(f;\mathfrak D) = \begin{pmatrix}
			\displaystyle\int_{\gamma_1} \dfrac{dx}{\sqrt{f(x)}} & \displaystyle\int_{\gamma_2} \dfrac{dx}{\sqrt{f(x)}}  & \displaystyle\int_{\gamma_3} \dfrac{dx}{\sqrt{f(x)}} \\
			\displaystyle\int_{\gamma_1} \dfrac{xdx}{\sqrt{f(x)}} & \displaystyle\int_{\gamma_2} \dfrac{xdx}{\sqrt{f(x)}} &\displaystyle\int_{\gamma_3} \dfrac{xdx}{\sqrt{f(x)}}
		\end{pmatrix},
	\end{equation}
	\begin{equation}\label{eqEMatrix}
		\mathfrak E(f;\mathfrak D) = -
		\begin{pmatrix}
			\displaystyle\int_{\gamma_1} \dfrac{\rho_1(x)dx}{4\sqrt{f(x)}} & \displaystyle\int_{\gamma_2} \dfrac{\rho_1(x)dx}{4\sqrt{f(x)}}  & \displaystyle\int_{\gamma_3} \dfrac{\rho_1(x)dx}{4\sqrt{f(x)}} \\
			\displaystyle\int_{\gamma_1} \dfrac{\rho_2(x)dx}{4\sqrt{f(x)}} & \displaystyle\int_{\gamma_2} \dfrac{\rho_2(x)dx}{4\sqrt{f(x)}}  & \displaystyle\int_{\gamma_3} \dfrac{\rho_2(x)dx}{4\sqrt{f(x)}}
		\end{pmatrix},
	\end{equation}
	where $\rho_1(x) = f_3x + 2f_4 x^2 + 3f_5 x^3 + 4f_6x^4$, $\rho_2(x) = f_5x^2 + 2f_6x^3$, and the branch of the square root is chosen to be the same in all integrals. Note that these matrices are well-defined (up to a simultaneous sign change) for all polynomials $f \in \mathfrak P_6$ that are subordinate to circles $\mathfrak D$. Further, the sum of columns of $\mathfrak W(f;\mathfrak D)$ (or $\mathfrak E(f;\mathfrak D)$) is equal to $0$. Moreover, if $f$ is admissible, then all columns of $\mathfrak W(f;\mathfrak D)$ belong to $\Per_f$ and they generate a maximal isotropic (with respect to intersection pairing) subgroup in $\Per_f$. In particular, there exists a symplectic basis $a_1,a_2,b_1,b_2 \in \Per_f$ such that $a_1$ and $a_2$ are any two given columns of $\mathfrak W(f;\mathfrak D)$. Moreover, if $f$ is admissible, then each column of the matrix $\mathfrak E(f;\mathfrak D)$ is the value of $\eta^f$ calculated at the corresponding column of $\mathfrak W(f;\mathfrak D)$.

	\begin{proposition}\label{propPeriodsRelationWithSubordination}
		Let $f$ be an admissible polynomial subordinate to $\mathfrak D = (D_1, D_2, D_3)$. Then we have
		\[
		\mathfrak W(f; \mathfrak D) = \pm \mathfrak W(H_{\mathfrak D}(f); \mathfrak D).
		\]
	\end{proposition}
	\begin{remark}
		The reason for the appearance of an unknown sign in Proposition~\ref{propPeriodsRelationWithSubordination} is due to the indeterminacy of the square root branches. By careful choice of the branches (see the proof below) it is possible to remove this indeterminacy.
	\end{remark}
	\begin{proof}
		Let us first state the idea of the proof. In~\cite[Section~3.1]{BostMestre} it is proved that the equality $\mathfrak W(f; \mathfrak D) = \mathfrak W(H_{\mathfrak D}(f); \mathfrak D)$ holds if $f$ is admissible, all the roots and the leading coefficient of $f$ are real and the branches of the square root are reasonably chosen. We extend this result to our generality by the uniqueness of analytic continuation. Now we proceed to the proof.
		
		By applying a suitable M\"obius transformation and shrinking the discs we can assume without loss of generality that $\overline{D_j} \subset \CC$ and that $D_j \cap \RR \ne \emptyset$ for all $j \in \{1,2,3\}$. Moreover, we can assume that $f$ is monic. Let $\mathbf{D} = D_1^2 \times D_2^2 \times D_3^2$, so $\mathbf{D}$ is a polydisc in $\CC^6$. For all $e \in \mathbf{D}$ we denote by $f_e$ the monic polynomial of degree $6$, whose roots are precisely the components of $e$ counting multiplicity. For all $e \in \mathbf{D}$ by $\sqrt{f_e}$ we denote the branch of the square root of $f_e$ defined in a neighborhood of $\CC\PP(1) \setminus (D_1 \cup D_2 \cup D_3)$ such that
		\[
		\lim_{|z| \to +\infty} \frac{\sqrt{f_e(z)}}{z^3} = 1.
		\]
		With this choice of the square root branch the function $e \mapsto \mathfrak W(f_e; \mathfrak D)$ is a matrix-valued analytic function on the polydisc $\mathbf{D}$. Further, let $k(e)$ denote the leading coefficient of the polynomial $H_{\mathfrak D}(f_e)$. It is easy to see that $k(e) \ne 0$ for all $e \in \mathbf{D}$ and $k(e) > 0$ if $e \in \mathbf{D} \cap \RR^6$. Thus, we can choose the square root branch of the function $k$ on $\mathbf{D}$ such that $\sqrt{k(e)} > 0$ for all $e \in \mathbf{D} \cap \RR^6$. Finally, for all $e \in \mathbf{D_e}$ we choose a square root branch of $H_{\mathfrak D}(f_e)$ such that 
		\[
		\lim_{|z| \to +\infty} \frac{\sqrt{H_{\mathfrak D}(f_e)(z)}}{\sqrt{k(e)}z^3} = 1.
		\]
		With this choice the function $e \mapsto \mathfrak W(H_{\mathfrak D}(f_e); \mathfrak D)$ also is an analytic function on $\mathbf{D}$. Moreover, with foregoing choices from~\cite[Section~3.1]{BostMestre} it is evident that the equality
		\[
		\mathfrak W(f_e; \mathfrak D) = \mathfrak W(H_{\mathfrak D}(f_e); \mathfrak D)
		\]
		holds for all $e \in \mathbf{D} \cap \RR^6$ such that $f_e$ is admissible. Thus, by identity theorem~\cite[\S~1.1]{Grauert} we conclude that this equality holds everywhere on $\mathbf{D}$.
	\end{proof}
	By combining Proposition~\ref{propPeriodsRelationWithSubordination} with Proposition~\ref{propEtaTransformation} and the formula~\eqref{eqChiFormula} we obtain the following result.
	\begin{corollary}
		Let $f$ be an admissible polynomial subordinate to $\mathfrak D = (D_1, D_2, D_3)$. If $f = p_1p_2p_3$ is the decomposition such that the roots of $p_j$ belong to $D_j$, then
		\begin{equation}\label{eqEMatrixTransform}
			\mathfrak E(f; \mathfrak D) = \\ \pm (2\mathfrak E(H_{\mathfrak D}(f);\mathfrak D) + \mathfrak H^{p_1,p_2,p_3}\mathfrak W(H_{\mathfrak D}(f); \mathfrak D)),
		\end{equation}
		where $\mathfrak H^{p_1,p_2,p_3}$ is given in~\eqref{eqChiFormula}.
	\end{corollary}
	
	To conclude this subsection we formulate the following fact.
	
	\begin{theorem}\label{thExistenceOfDiscs}
		Let $f$ -- be any admissible polynomial. Then there exist three disjoint open discs $D_1,D_2,D_3$ such that $f$ is subordinate to $\mathfrak D = (D_1,D_2,D_3)$. In other words, given any six-element subset $A \subset \CC\PP(1)$, there exist three disjoint open discs $D_1,D_2,D_3$, such that $A \cap D_j$ consists exactly of two points for all $j \in \{1,2,3\}$.
	\end{theorem}
	
	The proof of this theorem (even though it is elementary) is rather long and tedious, so it will be published in a separate paper.
	
	\subsection{Convergence of the iterations of the Richelot construction}
	In this subsection we analyze the behavior of the iterations of the Richelot construction subordinate to a triple of disks. To formulate the results we recall that a sequence $\{a_n\}_{n \in \NN}$ is said to {\it{converge quadratically}} if there exist constants $C, D> 0$ such that $|a_n - a| \le C\exp(-D2^n)$ for all $n \in \NN$, where $a = \lim a_n$. It is known that, if a sequence $\{a_n\}_{n \in \NN}$ fulfills the  inequalities $|a_{n+1}| \le C|a_n|^2$ for all $n$ with a suitable constant $C$ and converges to zero, then it converges quadratically. Moreover, the foregoing conditions are the most frequent source of quadratically convergent sequences, so often in literature they are used as the definition of quadratic convergence. The definition that we opted to give is more suitable for our considerations, as it behaves better with arithmetic operations. That is, we shall often use the fact that sums and products of quadratically convergent sequences also converge quadratically. More generally, quadratic convergence (in our definition) is preserved by substituting into functions that are holomorphic in a neighbourhood of the limit. Moreover, we shall use the fact the partial sums (resp. partial products) of sequences that quadratically converge to $0$ (resp. to $1$) also converge quadratically. Finally, we say that a sequence $\{a_n\}_{n \in \NN}$ converges {\it{strictly quadratically}}, if there exist positive constants $A,B,C,D$ such that $A\exp(-B2^n)\le|a_n - a| \le C\exp(-D2^n)$ for all $n$. This concept will be used in the next section, where we analyze the convergence of Kleinian functions of weight $2$ when performing iterations of the Richelot construction.
	
	\begin{theorem}\label{thConvergenceOfPolynomials}
		Let $f$ be an admissible polynomial subordinate to $\mathfrak D = (D_1, D_2, D_3)$, and consider the sequence $\left\{f^{(n)}\right\}_{n = 0}^{\infty}$ of elements of $\mathfrak P_6$ defined by the rules  
		\begin{equation}\label{eqDefRichSeq}
			f^{(0)} = f,\;\;f^{(n+1)} = H_{\mathfrak D}\left(f^{(n)}\right).
		\end{equation}
		Then the sequence $f^{(n)}$ quadratically converges to a polynomial that has three double roots, one in each circle $D_j$, $j = 1,2,3$.
	\end{theorem}
	
	We prove Theorem~\ref{thConvergenceOfPolynomials} in several steps. The most frequent idea that we use in the proof of it is the projective invariance, which follows from Proposition~\ref{propBasicBasic}~\ref{BasicBasicv} and Proposition~\ref{propBasicProperties}~\ref{BasicPropertiesiv}.
	
	\begin{lemma}\label{lemStationaryIfDouble}
		Let $\mathfrak D = (D_1, D_2, D_3)$ be a triple of disjoint open disks and let $f \in \mathfrak P_6$ be subordinate to $\mathfrak D$. Assume that $k \in \{1,2,3\}$ and that $f$ has a double root $a \in D_k$. Then $H_{\mathfrak D}(f)$ has the same double root $a$.
	\end{lemma}
	\begin{proof}
		Without loss of generality we assume that $k = 1$. It is straightforward to check from the definition of $[\cdot,\cdot]$ that if $p$ has a double root $a$, then $a$ is also a root of $[p,q]$ for all $q$. Now consider the decomposition $f = p_1p_2p_3$, where both roots of $p_j$ belong to $D_j$, $j = 1,2,3$. Then $p_1$ has the double root $a$, so $a$ is a root of both $[p_1, p_2]$ and $[p_3,p_1]$. Thus, $a$ is a root of $H_{\mathfrak D}(f)$ of multiplicity at least two. Since by Proposition~\ref{propDiskSubordination} it follows that four roots of $H_{\mathfrak D}(f)$ belong to $D_2 \cup D_3$, it follows that $a$ is the root of $H_{\mathfrak D}(f)$ of multiplicity exactly two.
	\end{proof}
	
	\begin{lemma}\label{lemConvergenceIfSmall}
		Let $\mathfrak D = (D_1, D_2, D_3)$ be disjoint open disks and let $F_j \subset D_j$ be any closed disk for $j = 1,2,3$. Then the following statements hold.
		\begin{enumerate}[label=(\roman*)]
			\item\label{ConvergenceIfSmalli} Assume that $f \in \mathfrak P_6$ is a polynomial such that $F_j$ contains exactly two roots of $f$ for some $j \in \{1,2,3\}$. Then $F_j$ contains exactly two roots of $H_{D_1,D_2,D_3}(f)$.
			\item\label{ConvergenceIfSmallii} Assume now that $D_j \subset \CC$ for $j = 1,2,3$ (i.e. none of the disks $D_j$ contains the infinity). For $f \in \mathfrak P_6$ that is subordinate to $D_1,D_2,D_3$ let $\delta_j(f) = |a - b|$, where $a$ and $b$ are the roots of $f$ that belong to $D_j$. Finally, fix $l \in \{1,2,3\}$. Then there exists $\varepsilon > 0$ such that the following statement holds. 
			For all $f \in \mathfrak P_6$ such that $F_j$ contains exactly two roots of $f$ for all $j \in \{1,2,3\}$ the condition $\delta_l(f) \le \varepsilon$ implies that both roots of $f^{(n)}$ that belong to $F_l$ converge quadratically to a common limit, where the sequence $f^{(n)}$ is defined by~\eqref{eqDefRichSeq}. Moreover, if in addition $\delta_l(f) \ne 0$, then the convergence of the sequence $\delta_l\left(f^{(n)}\right)$ to zero is strictly quadratic.
			\item\label{ConvergenceIfSmalliii} Assume again that $D_j \subset \CC$ for $j = 1,2,3$ and let $\varepsilon_1, \varepsilon_2,\varepsilon_3 > 0$ be the numbers from~\ref{ConvergenceIfSmallii} applied to $l = 1,2,3$ respectively. Assume that $F_j$ contains exactly two roots of a polynomial $f \in \mathfrak P_6$ and the inequality $\delta_j(f) \le \varepsilon_j$ holds for all $j \in \{1,2,3\}$, then the sequence $f^{(n)}$ defined in~\eqref{eqDefRichSeq} quadratically converges to a polynomial that has three double roots, one in each disk $F_j$, $j = 1,2,3$.
		\end{enumerate}
	\end{lemma}
	\begin{proof}
		Statement~\ref{ConvergenceIfSmalli} easily follows from Proposition~\ref{propDiskSubordination}. Assume for simplicity that $j = 1$. Then it is possible to find a sequence $\left\{D_1^{(n)}\right\}_{n \in \NN}$ of open disks such that $D_1^{(1)} = D_1$, $D_1^{(1)} \supset D_1^{(2)} \supset \dots$ and $\bigcap_n D_1^{(n)} = F_1$. Now if $F_1$ contains exactly two roots of $f$, then $f$ is subordinate to the triple $\mathfrak D^{(n)} = (D_1^{(n)}, D_2, D_3)$ for all $n$. Thus, $H_{\mathfrak D}(f) = H_{\mathfrak D^{(n)}}(f)$ is again subordinate to $\mathfrak D^{(n)}$ by Proposition~\ref{propDiskSubordination}. Thus, $F_1 = \bigcap_n D_1^{(n)}$ contains exactly two roots of $H_{\mathfrak D}(f)$.
		
		To prove statement~\ref{ConvergenceIfSmallii} it clearly suffices to prove the existence of $\varepsilon_1$. Consider $f \in \mathfrak P_6$ such that $F_j$ contains exactly two roots of $f$ for $j = 1,2,3$. Assume for a moment that $f$ is monic, so 
		\[
		f(x) = \left(x - e_1^{(1)}\right)\left(x - e_2^{(1)}\right)\left(x - e_1^{(2)}\right)\left(x - e_2^{(2)}\right)\left(x - e_1^{(3)}\right)\left(x - e_2^{(3)}\right),
		\]
		where $e_1^{(j)}, e_2^{(j)} \in F_j$ for $j = 1,2,3$. Let $p_j(x) = \left(x - e_1^{(j)}\right)\left(x - e_2^{(1)}\right)$, so $f = p_1p_2p_3$ and 
		\[
		H_{\mathfrak D}(f) = \frac{1}{4\Delta(p_1,p_2,p_3)}[p_2,p_3][p_3,p_1][p_1,p_2].
		\]
		Let $a_1$ and $b_1$ be the roots of $[p_1,p_2]$ and $[p_3, p_1]$ respectively that belong to $D_1$ (by~\ref{ConvergenceIfSmalli} they in fact belong to $F_1$). Also let $a_2$ and $b_2$ denote the remaining roots of $[p_1,p_2]$ and $[p_3,p_1]$ respectively (so $a_2 \in F_2$ and $b_2 \in F_3$). By computing the leading coefficients of these polynomials explicitly we find
		\[
		\begin{split}
			[p_1,p_2](x) = \left(e_1^{(1)} + e_2^{(1)} - e_1^{(2)} - e_2^{(2)}\right)(x - a_1)(x - a_2),\\
			[p_3,p_1](x) = \left(e_1^{(3)} + e_2^{(3)} - e_1^{(1)} - e_2^{(1)}\right)(x - b_1)(x - b_2).
		\end{split}
		\]
		From these formulas and~\eqref{eqDiscrAndResRootFormula} we find that 
		\begin{multline*}
			\Res([p_1,p_2], [p_3,p_1]) = \left(e_1^{(1)} + e_2^{(1)} - e_1^{(2)} - e_2^{(2)}\right)^2\left(e_1^{(3)} + e_2^{(3)} - e_1^{(1)} - e_2^{(1)}\right)^2 \times \\
			(a_1 - b_1)(a_2 - b_1)(a_1 - b_2)(a_2 - b_2).
		\end{multline*}
		On the other hand from Proposition~\ref{propBasicProperties}~\ref{BasicPropertiesiii} we obtain
		\[
		\Res([p_1,p_2], [p_3,p_1]) = \Delta(p_1,p_2,p_3)^2\Discr(p_1) = \Delta(p,q,r)^2\left(e_1^{(1)} - e_2^{(1)}\right)^2.
		\]
		Combining the expressions for the resultant yields
		\begin{multline}\label{eqRootDiff}
			(a_1 - b_1) = \\ \frac{\Delta(p_1,p_2,p_3)^2\left(e_1^{(1)} - e_2^{(1)}\right)^2}{(a_2 - b_1)(a_1 - b_2)(a_2 - b_2)\left(e_1^{(1)} + e_2^{(1)} - e_1^{(2)} - e_2^{(2)}\right)^2\left(e_1^{(3)} + e_2^{(3)} - e_1^{(1)} - e_2^{(1)}\right)^2}.
		\end{multline}
		For distinct $j,k \in \{1,2,3\}$ let $d(j,k) = \inf \{|u - v|: u \in F_j, v \in F_k\}$ and let
		\[
		C = \sup\{ \Delta(p_1,p_2,p_3): p_j\text{ is monic and both its roots lie in } F_j\text{ for } j = 1,2,3\}.
		\]
		With these definitions we can write
		\begin{equation}\label{eqQuadraticConv}
			\delta_1(H_{\mathfrak D}(f)) \le \delta_1(f)^2 \frac{C^2}{16d(1,2)^3d(1,3)^3d(2,3)}.
		\end{equation}
		Thus, if we put $\varepsilon_1 =  8d(1,2)^3d(1,3)^3d(2,3)/C^2$ and assume that $\delta_1(f) \le \varepsilon_1$, then $\delta_1(H_{D_1,D_2,D_3}(f)) \le \delta_1(f)/2$. That is the condition $\delta_1(f) \le \varepsilon_1$ implies $\delta_1(f^{(n)}) \to 0$ when $n \to +\infty$. From~\eqref{eqQuadraticConv} it follows that this convergence is quadratic. Now assume that $\delta_l(f) \ne 0$. Then from~\eqref{eqRootDiff} follows that $\delta_l(f^{(n)}) \ne 0$ for all $n$. Moreover, if
		\[
		\tilde C = \inf\{ \Delta(p_1,p_2,p_3): p_j\text{ is monic and both its roots lie in } F_j\text{ for } j = 1,2,3\},
		\]
		then $\tilde C \ne 0$ and $\tilde{d}(j,k) = \sup \{|u - v|: u \in F_j, v \in F_k\}$
		\[
		\delta_1(H_{\mathfrak D}(f)) \ge \delta_1(f)^2 B,\text{ where } B = \frac{\tilde C^2}{16 \tilde d(1,2)^3\tilde d(1,3)^3\tilde d(2,3)}.
		\]
		This implies that $\ln(\delta_1(f^{(n+1)})) \ge 2\ln(\delta_1(f^{(n+1)})) + \ln(B)$, so the convergence is strictly quadratic. From the statement~\ref{ConvergenceIfSmallii} it remains to prove that the pair of roots of $f^{(n)}$ in $F_1$ converges to a common limit. To do this denote by $a_1^{(n)}$ and $b_1^{(n)}$ the of roots of $f^{(n)}$ in $F_1$. If there is $m$ such that $a_1^{(m)} = b_1^{(m)}$, then both sequences are stationary starting from $m$ by Lemma~\ref{lemStationaryIfDouble}, so this case is trivial. Thus, we assume that $\delta_1(f^{(n)}) > 0$ for all $n$ and let \[Y_n = \left\{v \in \CC: \left|v - \left(a_1^{(n)} + b_1^{(n)}\right)/2\right| \le \delta_1\left(f^{(n)}\right)\right\}.\] Since $\delta_1(f^{(n)}) \to 0$ and the center of the disk $Y_n$ belongs to $F_1$ we can find $m \in \NN$ such that $Y_n \subset D_1$ for all $n \ge m$. Then~\ref{ConvergenceIfSmalli} implies that $a_1^{(k)}, b_1^{(k)} \in Y_n$ if $k \ge n \ge m$. Finally, let $Z_n = \bigcap_{n \ge k \ge m} Y_k$. Then $a_1^{(n)}, b_1^{(n)} \in Z_n$ for $n \ge m$ and the sequence $Z_n$ monotonically decreases, i.e. $Z_m \supset Z_{m+1} \supset \dots$. Moreover, $\mathrm{diam}(Z_n)$ is estimated from above by $\delta_1(f^{(n)})$ so $\mathrm{diam}(Z_n) \to 0$ quadratically, when $n \to +\infty$. Now it is evident that the roots $a_1^{(n)}$ and $b_1^{(n)}$ converge quadratically to the unique point in the intersection $\bigcap_{n \ge m} Z_n$.
		
		Finally, we prove~\ref{ConvergenceIfSmalliii}. Consider $f \in \mathfrak P_6$ such that $F_j$ contains exactly two roots of $f$ and that $\delta_j(f) \le \varepsilon_j$ for $j = 1,2,3$. Let the sequence $f^{(n)}$ be defined by~\eqref{eqDefRichSeq} and write
		\begin{multline*}
			f^{(n)}(x) = \\ l_n\left(x - e_1^{(n,1)}\right)\left(x - e_2^{(n,1)}\right)\left(x - e_1^{(n,2)}\right)\left(x - e_2^{(n,2)}\right)\left(x - e_1^{(n,3)}\right)\left(x - e_2^{(n,3)}\right),
		\end{multline*}
		where $e_1^{(n,j)}, e_2^{(n,j)} \in F_j$. From~\ref{ConvergenceIfSmallii} it follows that $e_1^{(n,j)},e_2^{(n,j)} \to t_j \in F_j$ quadratically when $n \to +\infty$ for $j = 1,2,3$. Therefore, to prove~\ref{ConvergenceIfSmalliii} it remains to show that the sequence $\{l_n\}_{n \in \NN}$ converges quadratically to a non-zero limit. By a direct calculation it is easy to find that $l_{n+1} = q_{n+1}l_n$, where
		\begin{multline*}
			q_{n+1} = \frac{1}{4\Delta\left(p_1^{(n)}, p_2^{(n)}, p_3^{(n)}\right)}\left(e_1^{(n,1)} + e_2^{(n,1)} - e_1^{(n,2)} - e_2^{(n,2)}\right) \times \\ \left(e_1^{(n,3)} + e_2^{(n,3)} - e_1^{(n,1)} - e_2^{(n,1)}\right)\left(e_1^{(n,2)} + e_2^{(n,2)} - e_1^{(n,3)} - e_2^{(n,3)}\right),
		\end{multline*}
		where $p_j^{(n)}(x) = \left(x - e_1^{(n,j)}\right)\left(x - e_2^{(n,j)}\right)$. Hence, $l_{n} = q_nq_{n-1}\dots q_1l_0$, so it suffices to prove that the sequence $\{q_n\}_{n \in \NN}$ quadratically converges to $1$. It is easy to see that $q_n \ne 0$ for all $n$. Moreover, the sequence $\{q_n\}_{n \in \NN}$ is obtained by evaluating an analytic function of six variables at a sequence of points that quadratically converges, it follows that $\{q_n\}_{n \in \NN}$ quadratically converges to the value
		\[
		q = \frac{8(t_1 - t_2)(t_2 - t_3)(t_3 - t_1)}{4\Delta(T_1,T_2,T_3)},
		\]
		where $T_j(x) = (x - t_j)^2$, $j = 1,2,3$. An easy calculation shows that $q = 1$.
	\end{proof}
	\begin{lemma}\label{lemPartialConvergence}
		Let $\mathfrak D = (D_1, D_2, D_3)$ be a triple of disjoint open disks and let $f \in \mathfrak P_6$ be subordinate to disks $\mathfrak D$. Let the sequence $f^{(n)}$ be defined by~\eqref{eqDefRichSeq} and denote by $e^{(n,j)}_{1}, e^{(n,j)}_{2}$ be the roots of $f^{(n)}$ that belong to $D_j$. Then the following statements hold.
		\begin{enumerate}[label=(\roman*)]
			\item\label{PartialConvergencei} Assume that $e^{(0,j)}_{1} = e^{(0,j)}_{2}$ for $j = 2,3$. Then the sequences $\{e^{(n,1)}_{1}\}_{n \in \NN}$ and $\{e^{(n,1)}_{2}\}_{n \in \NN}$ converge to a common limit.
			\item\label{PartialConvergenceii} Assume that $e^{(0,3)}_{1} = e^{(0,3)}_{2}$. Then there exists $j \in \{1,2\}$ such that the sequences $\{e^{(n,j)}_{1}\}_{n \in \NN}$ and $\{e^{(n,j)}_{2}\}_{n \in \NN}$ converge to a common limit.
			\item\label{PartialConvergenceiii} Assume that $f$ is admissible (i.e. has no multiple roots). Then for at least two distinct $j \in \{1,2,3\}$ the sequences $\{e^{(n,j)}_{1}\}_{n \in \NN}$ and $\{e^{(n,j)}_{2}\}_{n \in \NN}$ converge to a common limit.
		\end{enumerate}
	\end{lemma}
	\begin{proof}
		Due to projective invariance we can apply a suitable M\"obius transformation to simplify the proof. That is, to prove~\ref{PartialConvergencei} we can assume without loss of generality that $e^{(0,2)}_{1} = e^{(0,2)}_{2} = 0$ and $e^{(0,3)}_{1} = e^{(0,3)}_{2} = \infty$. Note that Lemma~\ref{lemStationaryIfDouble} implies that $e^{(n,2)}_{1} = e^{(n,2)}_{2} = 0$ and $e^{(n,3)}_{1} = e^{(n,3)}_{2} = \infty$ for all $n$. Moreover, since $D_1$ in this case is an open disc that does not contain neither $0$ nor $\infty$, we can further assume that \[D_1 \subset \{z \in \CC: \mathrm{Re}\;z > 0\}.\] With these assumptions it is easy to verify that up to a transposition the pair $\left(e_1^{(n+1,1)}, e_2^{(n+1,1)}\right)$ coincides with the pair 
		\[
		\left( \frac{e_1^{(n,1)} + e_2^{(n,1)}}{2}, \frac{2e_1^{(n,1)}e_2^{(n,1)}}{e_1^{(n,1)} + e_2^{(n,1)}}\right).
		\]
		That is,~\ref{PartialConvergencei} is reduced to the fact that the iterations of arithmetic and harmonic mean of two numbers from the open right half-plane converge to a common limit. The proof we present here is based upon a trigonometric variable change (for a more detailed investigation for the case of real numbers see~\cite{NakamuraAHM}). Let $s_j = \sqrt{e_j^{(0,1)}}$ for $j = 1,2$, where the square root is chosen to be in the right half-plane. Then the number $s_1/s_2$ also belongs to the open right half-plane, so there exists $t \in \CC$ such that $\mathrm{Re}\; t > 0$ and
		\[
		\frac{s_1}{s_2} = \frac{1-e^{-2t}}{1+e^{-2t}} = \tanh(t).
		\]
		Then it is clear that 
		\[
		e_1^{(0,1)} = s_1s_2\tanh(t),\;\;e_2^{(0,1)} = s_1s_2\coth(t).
		\]
		Now from the formulas
		\[
		\coth(2t) = \frac{\tanh(t) + \coth(t)}{2},\;\;\tanh(2t) = \frac{2}{\tanh(t) + \coth(t)}
		\]
		it is clear that the pairs $\left(e_1^{(n+1,1)}, e_2^{(n+1,1)}\right)$ and $(s_1s_2\tanh(2^nt), s_1s_2\coth(2^nt))$ coincide up to a transposition. Since $\mathrm{Re}\; t > 0$ it is clear that $\tanh(2^nt),\coth(2^nt) \to 1$ when $n \to +\infty$. Thus,~\ref{PartialConvergencei} is proved.
		
		To prove~\ref{PartialConvergenceii} we can without loss of generality assume that $e^{(0,3)}_{1} = e^{(0,3)}_{2} = \infty$ and that $D_1 \subset \{z \in \CC: \mathrm{Re}\; z > 0\}$, $ D_2 \subset \{z \in \CC: \mathrm{Re}\; z < 0\}$. In this case the pairs $\left(e_1^{(n+1,1)}, e_2^{(n+1,1)}\right)$, $\left(e_1^{(n+1,2)}, e_2^{(n+1,2)}\right)$ up to the transpositions can be calculated by the rules
		\begin{equation}\label{eqPairsRecursion}
			\begin{gathered}
				e_1^{(n+1,1)} = \frac{e_1^{(n,1)} + e_2^{(n,1)}}{2},\;\;e_1^{(n+1,2)} = \frac{e_1^{(n,2)} + e_2^{(n,2)}}{2}, \\
				\left[p_1^{(n)}, p_2^{(n)}\right]\left(e_2^{(n+1,1)}\right) = \left[p_1^{(n)}, p_2^{(n)}\right]\left(e_2^{(n+1,2)}\right) = 0,
			\end{gathered}
		\end{equation}
		where $p_j^{(n)}(x) = \left(x - e_1^{(n,j)}\right)\left(x -e_2^{(n,j)}\right)$, $j = 1,2$. To simplify the notation we assume that pairs $\left(e_1^{(n+1,1)}, e_2^{(n+1,1)}\right)$ and $\left(e_1^{(n+1,2)}, e_2^{(n+1,2)}\right)$ are ordered in a way such that~\eqref{eqPairsRecursion} holds. Now we introduce
		\[
		a_n = \sqrt{\left(e_1^{(n,1)} - e_2^{(n,2)}\right)\left(e_2^{(n,1)} - e_1^{(n,2)}\right)},\;\;b_n = \sqrt{\left(e_1^{(n,1)} - e_1^{(n,2)}\right)\left(e_2^{(n,1)} - e_2^{(n,2)}\right)},
		\]
		where square roots are chosen to be in the right half-plane. We claim that $\{(a_n,b_n)\}$ is an AGM-sequence (see, e.g.~\cite{Cremona}), i.e.
		\[
		a_{n+1} = \frac{a_n + b_n}{2},\;\;b_{n+1}^2 = a_nb_n.
		\]
		Indeed, this can be verified by a direct calculation. Using Proposition~\ref{propBasicBasic}~\ref{BasicBasiciv} and that the leading coefficient of $\left[p_1^{(n)}, p_2^{(n)}\right]$ equals $\left(e_1^{(n,1)} + e_2^{(n,1)} - e_1^{(n,2)} - e_2^{(n,2)}\right)$, we get
		\begin{multline*}
			b_{n+1}^4 = \left(e_1^{(n+1,1)} - e_1^{(n+1,2)}\right)^2\left(e_2^{(n+1,1)} - e_2^{(n+1,2)}\right)^2 = \\
			\frac{1}{4}\left(e_1^{(n,1)} + e_2^{(n,1)} - e_1^{(n,2)} - e_2^{(n,2)}\right)^2\left(e_2^{(n+1,1)} - e_2^{(n+1,2)}\right)^2 =\\ \frac{1}{4}\Discr\left(\left[p_1^{(n)}, p_2^{(n)}\right]\right) = 
			\Res\left(p_1^{(n)}, p_2^{(n)}\right) = a_n^2 b_n^2.
		\end{multline*}
		Clearly, this means that $b_{n+1}^2 = a_n b_n$ due to the choice of square roots. Now by using the already established formula for $b_{n+1}^2$ and explicit formulas to compute $e_2^{(n,1)}e_2^{(n,2)}$ and $e_2^{(n,1)}+e_2^{(n,2)}$ through the coefficients of $\left[p_1^{(n)}, p_2^{(n)}\right]$ it is possible to calculate that
		 \begin{multline*}
		     a_{n+1}^2 = \left(e_1^{(n+1,1)} - e_2^{(n+1,2)}\right)\left(e_2^{(n+1,1)} - e_1^{(n+1,2)}\right) = \\ e_1^{(n+1,1)}e_2^{(n+1,1)} + e_2^{(n+1,2)}e_1^{(n+1,2)} - e_1^{(n+1,1)}e_1^{(n+1,2)} - e_2^{(n+1,2)}e_2^{(n+1,1)} = \\  \frac{1}{2}\left(e_1^{(n+1,1)} - e_1^{(n+1,2)}\right)\left(e_2^{(n+1,1)} - e_2^{(n+1,2)}\right) + \\ \frac{1}{2} \left(e_1^{(n+1,1)} + e_1^{(n+1,2)}\right)\left(e_2^{(n+1,1)} + e_2^{(n+1,2)}\right) - \\ \frac{1}{4}\left(e_1^{(n,1)} + e_2^{(n,1)}\right)\left(e_1^{(n,2)} + e_2^{(n,2)}\right) -
		     \\ \frac{e_1^{(n,1)}e_2^{(n,1)}\left(e_1^{(n,2)} + e_2^{(n,2)}\right) - \left(e_1^{(n,1)} + e_2^{(n,1)}\right)e_1^{(n,2)}e_2^{(n,2)}}{e_1^{(n,1)} + e_2^{(n,1)} - e_1^{(n,2)} - e_2^{(n,2)}}  = \\
		      -\frac{1}{4}\left(e_1^{(n,1)} + e_2^{(n,1)}\right)\left(e_1^{(n,2)} + e_2^{(n,2)}\right) + \frac{1}{2}b_{n+1}^2 - \\  \frac{e_1^{(n,1)}e_2^{(n,1)}\left(e_1^{(n,2)} + e_2^{(n,2)}\right) - \left(e_1^{(n,1)} + e_2^{(n,1)}\right)e_1^{(n,2)}e_2^{(n,2)}}{e_1^{(n,1)} + e_2^{(n,1)} - e_1^{(n,2)} - e_2^{(n,2)}} + \\
		      \frac{\left(e_1^{(n,1)}e_2^{(n,1)} - e_1^{(n,2)}e_2^{(n,2)}\right)\left(e_1^{(n,2)} + e_2^{(n,2)} + e_1^{(n,1)} + e_2^{(n,1)}\right)}{2\left(e_1^{(n,1)} + e_2^{(n,1)} - e_1^{(n,2)} - e_2^{(n,2)}\right)} = \\ -\frac{1}{4}\left(e_1^{(n,1)} + e_2^{(n,1)}\right)\left(e_1^{(n,2)} + e_2^{(n,2)}\right) + \frac{1}{2}b_{n+1}^2 + \frac{e_1^{(n,1)}e_2^{(n,1)} + e_1^{(n,2)}e_2^{(n,2)}}{2} = \\ \frac{a_n^2 + b_n^2 + 2b_{n+1}^2}{4} = \left(\frac{a_n + b_n}{2}\right)^2.
		 \end{multline*}
		Again by choice of square roots this implies that $a_{n+1} = (a_n + b_n)/2$. By~\cite[Proposition~1]{Cremona} the sequences $\{a_n\}$ and $\{b_n\}$ converge to a common limit. It is easy to see that this limit is not equal to zero (since we can shrink the disks $D_1$ and $D_2$ slightly, so the initial roots are still inside those disks, but now the distance from the roots in distinct disks is separated from $0$). Thus, the cross-ratio
		\[
		\left. \frac{\left(e_1^{(n,1)} - e_2^{(n,2)}\right)}{\left(e_1^{(n,1)} - e_1^{(n,2)}\right)} \middle/ \frac{\left(e_2^{(n,1)} - e_2^{(n,2)}\right)}{\left(e_2^{(n,1)} - e_1^{(n,2)}\right)} \right.
		\]
		converges to $1$ when $n \to +\infty$. This implies that both values $\left|e_1^{(n,1)} - e_2^{(n,1)}\right|$ and $\left|e_1^{(n,2)} - e_2^{(n,2)}\right|$ cannot be simultaneously separated from $0$ when $n \to +\infty$. Thus, Lemma~\ref{lemConvergenceIfSmall}~\ref{ConvergenceIfSmallii} implies the statement~\ref{PartialConvergenceii}.
		
		Finally, we prove~\ref{PartialConvergenceiii}. Here we can without loss of generality assume that $D_1,D_2,D_3 \subset \CC$. Since $f$ is admissible it follows that $f^{(n)}$ is admissible for all $n = 0,1,\dots$. We choose a symplectic basis $a_1,a_2,b_1,b_2 \in \Per_f$ such that $a_1$ and $a_2$ are the first and the second column of $\mathfrak W(f;\mathfrak D)$ respectively. From Proposition~\ref{propPeriodsRelationWithSubordination} and Proposition~\ref{propRichProperties}~\ref{Richii} it is clear that $a_1,a_2,2^nb_1,2^nb_2$ is a symplectic basis in $\Per_{f^{(n)}}$ for all $n$. Let $A$ (resp. $B$) be the matrices whose columns are $a_1$ and $a_2$ (resp. $b_1$ and $b_2$) and define $\Omega = A^{-1}B$. From Thomae formula~\cite[Theorem~III.8.1]{mumfordII} it follows that for all $(j,k) = (1,2),(1,3),(2,3)$ the equalities
		\[
		\begin{gathered}
			\frac{\left(e_1^{(n,j)} - e_2^{(n,j)}\right)^2\left(e_1^{(n,k)} - e_2^{(n,k)}\right)^2}{\left(e_1^{(n,j)} - e_1^{(n,k)}\right)\left(e_1^{(n,j)} - e_2^{(n,k)}\right)\left(e_2^{(n,j)} - e_1^{(n,k)}\right)\left(e_2^{(n,j)} - e_2^{(n,k)}\right)} = \\ \left(\frac{\theta \begin{bmatrix}
					\beta_{jk} \\ 0
				\end{bmatrix}(0,2^n\Omega)\;\;\theta \begin{bmatrix}
					\beta_{jk} \\ \alpha_j + \alpha_k
				\end{bmatrix}(0,2^n\Omega)}{\theta \begin{bmatrix}
					0 \\ \alpha_j
				\end{bmatrix}(0,2^n\Omega)\;\;\theta \begin{bmatrix}
					0 \\ \alpha_k
				\end{bmatrix}(0,2^n\Omega)}\right)^4, \end{gathered}\]
		hold, where we used the following notation:
		\[
		2\alpha_1 = \begin{pmatrix} 1 \\ 0 \end{pmatrix},\;\;2\alpha_2 = \begin{pmatrix} 0 \\ 1 \end{pmatrix},\;\;2\alpha_3 = \begin{pmatrix} 1 \\ 1 \end{pmatrix},\;\;2\beta_{12} = \begin{pmatrix} 1 \\ 1 \end{pmatrix},\;\;2\beta_{13} = \begin{pmatrix} 1 \\ 0 \end{pmatrix},\;\;2\beta_{23} = \begin{pmatrix} 0 \\ 1 \end{pmatrix}.
		\]
		From the definition of theta functions with characteristics it is easy to see that for all $\alpha \in \RR^2$ and  $\beta \in \{0, 1/2\}^2$ when $n \to +\infty$ we have
		\[
		\theta\begin{bmatrix} \beta \\ \alpha \end{bmatrix}(0,2^n\Omega) \to 0,\text{ if }\beta \ne 0 \text{ and }\theta\begin{bmatrix} \beta \\ \alpha \end{bmatrix}(0,2^n\Omega) \to 1,\text{ if }\beta = 0. 
		\]
		Thus, we can conclude that the values 
		\[
		\begin{gathered}
			\left(e_1^{(n,1)} - e_2^{(n,1)}\right)^2\left(e_1^{(n,2)} - e_2^{(n,2)}\right)^2,\;\;\left(e_1^{(n,1)} - e_2^{(n,1)}\right)^2\left(e_1^{(n,3)} - e_2^{(n,3)}\right)^2,\\ \left(e_1^{(n,2)} - e_2^{(n,2)}\right)^2\left(e_1^{(n,3)} - e_2^{(n,3)}\right)^2
		\end{gathered}
		\]
		converge to $0$ when $n \to +\infty$. The statement~\ref{PartialConvergenceiii} immediately follows in view of Lemma~\ref{lemConvergenceIfSmall}~\ref{ConvergenceIfSmallii}.
	\end{proof}
	\begin{proof}[Proof of Theorem~\ref{thConvergenceOfPolynomials}.]
		We again assume that $D_1,D_2,D_3 \subset \CC$. Fix $f \in \mathfrak P_6$ that is subordinate to $\mathfrak D = (D_1,D_2,D_3)$ and find closed disks $F_j \subset D_j$ such that $F_j$ contains both roots of $f$ that belong to $D_j$ for $j = 1,2,3$.
		
		Now consider the topological space $Y = F_1^{(2)} \times F_2^{(2)} \times F_3^{(2)}$, where the superscript $(2)$ means the symmetric square. That is, elements of $Y$ are the triples $(y_1,y_2,y_3)$, where each $y_j$ is the unordered pair of elements of $F_j$, $j = 1,2,3$. For each $y \in X$ define $g_y \in \mathfrak P_6$ to be the monic polynomial, whose roots are precisely all the components of $y_1,y_2,$ and $y_3$ counting multiplicity. By Lemma~\ref{lemConvergenceIfSmall}~\ref{ConvergenceIfSmalli} we can define a continuous map $H:Y \to Y$ such that for all $(y_1,y_2,y_3)$ we have $H(x) = (z_1,z_2,z_3)$, where $z_j$ is the unordered pair that consists of roots of $H_{\mathfrak D}(g_y)$ that belong to $D_j$. If by $y^{(n)}$ we denote the roots of $f^{(n)}$ naturally organized in three unordered triples, then we have $y^{(n+1)} = H(y^{(n)})$. 
		
		With the foregoing preparation we are finally ready to finish the proof. Let $A$ denote the closure of the set $\{y^{(n)}: n = 0,1,\dots\}$ in the space $Y$. We shall call an unordered pair {\it{diagonal}} if its elements are equal. From Lemma~\ref{lemConvergenceIfSmall}~\ref{ConvergenceIfSmalliii} it is evident that the statement of Theorem~\ref{thConvergenceOfPolynomials} is true, if $A$ contains a triple $y = (y_1,y_2,y_3)$ in which all pairs $y_j$ are diagonal. Indeed, in this case for all $\varepsilon > 0$ we can find $n \in \NN$ such that in all pairs $y^{(n)}_j$ the distance between components is at most $\varepsilon$. Now we claim that $A$ contains a triple $y = (y_1,y_2,y_3)$ in which at least two pairs are diagonal. Indeed, if all pairs in $y^{(0)}$ are not diagonal, then by Lemma~\ref{lemPartialConvergence}~\ref{PartialConvergenceiii} at least two of three sequences $\{y^{(n)}_1\}_{n \in \NN}$, $\{y^{(n)}_2\}_{n \in \NN}$, and $\{y^{(n)}_3\}_{n \in \NN}$ converge to a diagonal pair. So it remains to extract a subsequence such that the third pair converges. A similar argument involving Lemma~\ref{lemPartialConvergence}~\ref{PartialConvergenceii} works if one of the pairs of $y^{(0)}$ is diagonal. Finally, if two or more pairs in $y^{(0)}$ are diagonal, then the statement is trivial. Thus, $A$ indeed contains a triple $y$ such that at least two of its pairs are diagonal. From the definition of $A$ it is clear that $A$ is invariant under the mapping $H$, i.e. $H(A) \subset A$. Therefore, $A$ contains the limit of the sequence $\{H^n(y)\}_{n \in \NN}$, which by Lemma~\ref{lemPartialConvergence}~\ref{PartialConvergencei} exists and is a triple in which all pairs are diagonal. Thus, $A$ contains a triple in which all pairs are diagonal and Theorem~\ref{thConvergenceOfPolynomials} follows.
	\end{proof}
	
	\section{Limiting values of Kleinian functions}\label{sec:Limit}
	Here we prove that if $f^{(n)}$ is a sequence from Theorem~\ref{thConvergenceOfPolynomials}, then the corresponding canonical Kleinian functions of weight $2$ converge uniformly on compact subsets in $\CC^2$ when $n \to +\infty$. Moreover, we provide sufficient information about the limiting functions to derive explicit formulas for them. We begin with the formulation of the main result.
	
	\begin{theorem}\label{thPassingToTheLimit}
		Let $\mathfrak D = (D_1,D_2,D_3)$ be a triple disjoint open disks and assume that $f \in \mathfrak P_6$ is admissible and subordinate to $\mathfrak D$. Assume that the sequence $f^{(n)}$ is defined by~\eqref{eqDefRichSeq} and denote its limit by $g$. Let $M \in \CC^{2\times 2}$ and $L \in \CC^{3 \times 2}$ be the matrices satisfying
		\[
		\mathfrak E(g;\mathfrak D) = M \mathfrak W(g;\mathfrak D),\;\;L\mathfrak W(g;\mathfrak D) = \begin{pmatrix}
			1 & 0 & -1 \\
			0 & 1 & -1 \\
			1 & -1 & 0 
		\end{pmatrix}.
		\]
		Then as $n \to +\infty$ each of the functions $S^{f^{(n)}}$, $S_{jk}^{f^{(n)}}$ uniformly on compact sets in $\CC^2$ converges to a function of the form
		\[
		\exp(z^TMz)\left(\alpha + \beta \sin^2(\pi l_1z) + \gamma\sin^2(\pi l_2z) + \delta\sin^2(\pi l_3z)\right),
		\]
		where $l_j$ is the $j$th row of the matrix $L$ and $\alpha,\beta,\gamma,\delta$ are complex numbers.
	\end{theorem}
	
	It is easy to see that Theorem~\ref{thPassingToTheLimit} allows to explicitly compute the limits of the Kleinian functions of weight $2$ by computing matrices $\mathfrak W$ and $\mathfrak E$ for the limiting polynomial using residue calculus and by finding unknown coefficients to match the Taylor expansions at zero (formula~\eqref{eqTaylorExpansions}). We present the answer in the generic case, i.e. when the limiting polynomial does not have a root at infinity (the case when $g$ has a double root at infinity is covered by Theorem~\ref{thPassingToTheLimit} as well; moreover, the formulas in this case appear to be less cumbersome than those in the generic case). That is, assume the hypotheses of Theorem~\ref{thPassingToTheLimit} and that $g(x) = c(x - t_1)^2(x - t_2^2)(x - t_3)^2$ where $t_1,t_2,t_3 \in \CC$ and $c \ne 0$. Then the matrix $M$ can be computed as
	\begin{equation*}
		M_{c,t_1,t_2,t_3} = \frac{c}{2}\begin{pmatrix}
			t_1 t_2 t_3 (t_1 + t_2 + t_3) & -t_1 t_2 t_3 \\
			-t_1 t_2 t_3 & t_1 t_2 + t_1 t_3 + t_2 t_3
		\end{pmatrix}.
	\end{equation*}
	and for all $z \in \CC^2$ we have
	\begin{subequations}\label{eqSAllLimit}
		\begin{multline}\label{eqSLimit}
			S^{f^{(n)}}(z) \xrightarrow[n \to +\infty]{} -\frac{4 \exp(z^T M_{c,t_1,t_2,t_3}z)}{c (t_1 - t_2)^2(t_1 - t_3)^2 (t_2 - t_3)^2} \times \\ 
			\left( (t_1 - t_2)(t_1 - t_3)\sin^2\left(\frac{i\sqrt{c}}{2}(t_2 - t_3)(z_2 - t_1z_1)  \right) + \right. \\  (t_2 - t_1)(t_2 - t_3)\sin^2\left(\frac{i\sqrt{c}}{2}(t_1 - t_3)(z_2 - t_2z_1)  \right) + \\ \left.   (t_3 - t_1)(t_3 - t_2)\sin^2\left(\frac{i\sqrt{c}}{2}(t_1 - t_2)(z_2 - t_3z_1)  \right)   \right),
		\end{multline}
		\begin{multline}\label{eqS22Limit}
			S_{22}^{f^{(n)}}(z) \xrightarrow[n \to +\infty]{} -\frac{4 \exp(z^T M_{c,t_1,t_2,t_3}z)}{c (t_1 - t_2)^2(t_1 - t_3)^2 (t_2 - t_3)^2} \times \\ 
			\left( (t_1 - t_2)(t_1 - t_3)(t_2 + t_3)\sin^2\left(\frac{i\sqrt{c}}{2}(t_2 - t_3)(z_2 - t_1z_1)  \right) + \right. \\  (t_2 - t_1)(t_2 - t_3)(t_1 + t_3)\sin^2\left(\frac{i\sqrt{c}}{2}(t_1 - t_3)(z_2 - t_2z_1)  \right) + \\ \left.   (t_3 - t_1)(t_3 - t_2)(t_1 + t_2)\sin^2\left(\frac{i\sqrt{c}}{2}(t_1 - t_2)(z_2 - t_3z_1)  \right)   \right),
		\end{multline}
		\begin{multline}\label{eqS12Limit}
			S_{12}^{f^{(n)}}(z) \xrightarrow[n \to +\infty]{} \frac{4 \exp(z^T M_{c,t_1,t_2,t_3}z)}{c (t_1 - t_2)^2(t_1 - t_3)^2 (t_2 - t_3)^2} \times \\ 
			\left( (t_1 - t_2)(t_1 - t_3)t_2 t_3\sin^2\left(\frac{i\sqrt{c}}{2}(t_2 - t_3)(z_2 - t_1z_1)  \right) + \right. \\  (t_2 - t_1)(t_2 - t_3)t_1 t_3\sin^2\left(\frac{i\sqrt{c}}{2}(t_1 - t_3)(z_2 - t_2z_1)  \right) + \\ \left.   (t_3 - t_1)(t_3 - t_2)t_1 t_2\sin^2\left(\frac{i\sqrt{c}}{2}(t_1 - t_2)(z_2 - t_3z_1)  \right)   \right),
		\end{multline}
		\begin{multline}\label{eqS11Limit}
			S_{11}^{f^{(n)}}(z) \xrightarrow[n \to +\infty]{} \frac{2 \exp(z^T M_{c,t_1,t_2,t_3}z)}{(t_1 - t_2)^2(t_1 - t_3)^2 (t_2 - t_3)^2} \times \\ 
			\left( \frac{(t_1 - t_2)^2(t_1 - t_3)^2 (t_2 - t_3)^2}{2} + \right. \\  (t_1 - t_2)(t_1 - t_3)[t_1t_2 t_3(t_1 + t_2 + t_3) + t_2^2 t_3^2]\sin^2\left(\frac{i\sqrt{c}}{2}(t_2 - t_3)(z_2 - t_1z_1)  \right) +  \\  (t_2 - t_1)(t_2 - t_3)[t_1t_2 t_3(t_1 + t_2 + t_3) + t_1^2 t_3^2]\sin^2\left(\frac{i\sqrt{c}}{2}(t_1 - t_3)(z_2 - t_2z_1)  \right) + \\ \left.   (t_3 - t_1)(t_3 - t_2)[t_1t_2 t_3(t_1 + t_2 + t_3) + t_1^2 t_2^2]\sin^2\left(\frac{i\sqrt{c}}{2}(t_1 - t_2)(z_2 - t_3z_1)  \right)   \right).
		\end{multline}
	\end{subequations}
	
	The proof of Theorem~\ref{thPassingToTheLimit} relies on convergence properties of sequences $\{f_n\}_{n \in \NN}$, where $f_n \in R_2^{2^n\Omega}$ for some Riemann matrix $\Omega$. To state the main ingredient of the proof we need to distinguish a specific class of Riemann matrices.
	
	\begin{definition}\label{defQuasiReduced}
		We call a Riemann matrix $\Omega \in \CC^{2 \times 2}$ quasi-reduced if the statement
		\[
		k^T (\mathrm{Im}\;\Omega) k < m^T (\mathrm{Im}\;\Omega) m \;\;\forall m \in \ZZ^2 \text{ such that } m \ne \pm k \text{ and } m - k \in 2\ZZ^2
		\]
		holds for all $k \in E$, where
		\[
		E = \left\{ \begin{pmatrix}
			1 \\ 0
		\end{pmatrix},\;\begin{pmatrix}
			0 \\ 1
		\end{pmatrix},\;
		\begin{pmatrix}
			1 \\ -1
		\end{pmatrix}\right\}.
		\]
	\end{definition}
	
	That is, a Riemann matrix is quasi-reduced if each vector in $E$ has strictly minimal length in its coset modulo $2\ZZ^2$ (ignoring, of course, the opposite sign vector) with respect to Euclidean norm induced by $\mathrm{Im}\;\Omega$. The choice for the name comes from an analogy of this concept to the Minkowski reduction theory (see, e.g.~\cite[\S~V.4]{Igusa}).
	
	\begin{lemma}\label{lemThetaConvergence}
		Assume that $\Omega$ is a quasi-reduced Riemann matrix and let $\{c_n\}_{n \in \mathbb N}$ be any sequence of positive real numbers that converges to $+\infty$. Consider any sequence $\{f_n\}_{n \in \NN}$, where $f_n \in R_2^{c_n \Omega}$. Then the following statements hold.
		\begin{enumerate}[label=(\roman*)]
			\item\label{thetaconvergencei} The sequence $\{f_n\}$ converges uniformly on compact subsets in $\CC^2$ if only if the sequence $p_n$ converges, where $p_n$ is the order $2$ Taylor expansion of $f_n$ at zero (i.e. $f_n(z) = p_n(z) + \bar{o}(z^2)$.
			\item\label{thetaconvergenceii} If the sequence $\{f_n\}$ converges uniformly on compact subsets in $\CC^2$, and $f$ denotes the limit of this sequence, then \[f \in \Span\{1, \sin^2(\pi z_1), \sin^2(\pi z_2),\sin^2(\pi(z_1 - z_2))\}.\]
		\end{enumerate}
	\end{lemma}
	\begin{proof}
		We shall denote Fourier coefficients of a function $g$ by $g\langle\cdot\rangle$. From the definition of the space $R_2^\Omega$ is is easy to see that a function $g \in R_2^\Omega$ can be expressed as a Fourier series
		\[
		g(z) = \sum_{m \in \ZZ^2}g\langle m\rangle \exp\left(2\pi i m^Tz\right),
		\]
		where the coefficients satisfy the relations
		\begin{equation}\label{eqFourierRecursion}
			g\langle m + 2k\rangle = \exp\left(2\pi i(m+ 2k)^T\Omega(m+ 2k) - 2\pi im^T\Omega m\right)
		\end{equation}
		for all $m,k \in \ZZ^2$. With this preparation we can prove an auxiliary fact: the sequence $\{f_n\}_{n \in \NN}$ converges uniformly on compact subsets of $\CC^2$ if and only if the sequences of Fourier coefficients $\left\{f_n\langle k\rangle \right\}_{n \in \NN}$ converge for all $k \in E \cup \{0\}$. Indeed, if the sequence $\{f_n\}_{n \in \NN}$ converges, then all Fourier coefficients are convergent sequences, since Fourier coefficient can be restored from the function by integration. On the other hand, assume that the sequences $\left\{f_n\langle k\rangle \right\}_{n \in \NN}$ converge for all $k \in E \cup \{0\}$. From~\eqref{eqFourierRecursion} it is evident that
		\begin{multline}\label{eqFourierExpansion}
			f_n(z) = f_n\langle 0\rangle\sum_{k \in \ZZ^2}\exp\left(8\pi ic_nk^T\Omega k\right)\exp\left(2\pi i (2k)^Tz\right) + \\ \sum_{m \in E} f_n\langle m\rangle \sum_{k \in \ZZ^2}\exp\left[2\pi ic_n\left((m+ 2k)^T\Omega(m+ 2k) - m^T\Omega m\right)\right]\exp(2\pi i(2k + m)^Tz).
		\end{multline}
		It is easy to see that the condition $c_n \to +\infty$ implies that when $n \to +\infty$ we have
		\[
		\sum_{k \in \ZZ^2}\exp\left(8\pi ic_nk^T\Omega k\right)\exp\left(2\pi i (2k)^Tz\right) \to 1
		\]
		uniformly with respect to $z$ on compact sets in $\CC^2$. Using the fact that $\Omega$ is quasi-reduced we also find that
		\begin{multline*}
			\sum_{k \in \ZZ^2}\exp\left[2\pi ic_n\left((m+ 2k)^T\Omega(m+ 2k) - m^T\Omega m\right)\right]\exp(2\pi i(2k + m)^Tz) \to \\ \exp(2\pi i m^Tz) + \exp(-2\pi i m^Tz) = 2\cos(2\pi m^tz)
		\end{multline*}
		uniformly on compact sets in $\CC^2$ (that is, the only terms that survive passing to the limit correspond to $k = 0$ and $k = -m$). Thus, the auxiliary statement is proved. In the meantime we also have found a formula for the limit, namely
		\[
		\lim_{n \to +\infty} f_n(z) = \lim_{n \to +\infty} f\langle 0\rangle  + \sum_{m \in E}\left(\lim_{n \to +\infty} f_n\langle m \rangle\right)2\cos(2\pi m^tz).
		\]
		The statement~\ref{thetaconvergenceii} is, therefore, proved.
		
		Now we prove~\ref{thetaconvergencei}. Clearly, if the sequence $\{f_n\}_{n \in \NN}$ converges uniformly in a neighbourhood of zero, then the Taylor expansions also constitute a convergent sequence. Thus, it remains to prove the converse. Let $p_n = p_n^{(0)} + p_n^{(11)}z_1^2 + p_n^{(12)}z_1z_2 + p_n^{(22)}z_2^2$ and assume that all the coefficients converge when $n \to +\infty$. From~\eqref{eqFourierExpansion} it is easy to calculate coefficients of $p_n$ by using Fourier coefficients $f_n \langle 0\rangle$ and $f_n\langle k\rangle$, $k \in E$. Indeed, if we put
		\[
		\alpha_1 = \begin{pmatrix}
			1 & 0
		\end{pmatrix}^T,\;\;\alpha_2 = \begin{pmatrix}
			0 & 1
		\end{pmatrix}^T,\;\;\alpha_3 = \begin{pmatrix}
			1 & -1
		\end{pmatrix}^T,
		\]
		then we can calculate that
		\[
		\begin{pmatrix}
			p_n^{(0)} \\ p_n^{(11)} \\ p_n^{(12)} \\ p_n^{(22)}
		\end{pmatrix} = A_n\begin{pmatrix}
			f_n\langle 0\rangle \\ f_n\langle \alpha_1 \rangle \\
			f_n\langle \alpha_2 \rangle \\ f_n\langle \alpha_3 \rangle
		\end{pmatrix}
		\]
		with a suitable matrix $A_n$ (which is universal for all functions in $R_2^{c_n\Omega}$). The formula for entries of $A_n$ is quite easy to obtain but we shall not give it here, since all that we are interested in is the limit of $A_n$ when $n \to +\infty$, which is easily computed as
		\[
		\lim_{n \to +\infty} A_n = \begin{pmatrix}
			1 & 2 & 2 & 2 \\
			0 & -4\pi^2 & 0 & -4\pi^2\\
			0 & 0 & 0 & 8\pi^2\\
			0 & 0 & -4\pi^2 & -4\pi^2
		\end{pmatrix}.
		\]
		Thus, we conclude that the limit of $A_n$ is an invertible matrix, so for large $n$ all matrices $A_n$ are invertible and, therefore, for large $n$ we can write that \[
		\begin{pmatrix}
			f_n\langle 0\rangle \\ f_n\langle \alpha_1 \rangle \\
			f_n\langle \alpha_2 \rangle \\ f_n\langle \alpha_3 \rangle
		\end{pmatrix} = A_n^{-1}
		\begin{pmatrix}
			p_n^{(0)} \\ p_n^{(11)} \\ p_n^{(12)} \\ p_n^{(22)}
		\end{pmatrix}
		\]
		Since on the right-hand side of this equation all sequences converge (note that the matrix inversion is continuous), we obtain that Fourier coefficients of functions $f_n$ converge when $n \to +\infty$. Thus, using our auxiliary statement we get~\ref{thetaconvergencei}.
	\end{proof}
	
	The next step is to show applicability of Lemma~\ref{lemThetaConvergence} to the situation of Theorem~\ref{thPassingToTheLimit}. This is the subject of the following lemma.
	
	\begin{lemma}\label{lemQuasiReducedIfSubordinate}
		Let $\mathfrak D = (D_1,D_2,D_3)$ be a triple of disjoint open disks and assume that $f \in \mathfrak P_6$ is admissible and subordinate to $\mathfrak D$. Let $a_1,a_2,b_1,b_2$ be a symplectic basis in $\Per_f$ such that $a_1$ and $a_2$ are any two columns of $\mathfrak W(f;\mathfrak D)$. Then the Riemann matrix $\Omega = A^{-1}B$ is quasi-reduced, where $A$ and $B$ are formed from columns $a_1, a_2$ and $b_1, b_2$ respectively.
	\end{lemma}
	
	To prove this fact we need a couple of auxiliary statements.
	\begin{lemma}\label{lemIntegralEstimates}
		Let $\mathbb D \subset \CC$ denote the closed unit disc.
		\begin{enumerate}[label=(\roman*)]
			\item\label{IntegralEstimatesi} There exists a constant $\Phi_1 > 0$ such that for all $\varepsilon > 0$ and for all $a, b \in \CC$ such that $|a|,|b| \le \varepsilon$ and $|a - b| \ge \varepsilon/2$ the inequality
			\[
			\int_{\varepsilon \mathbb D} \frac{idx \wedge d\bar{x}}{|x - a||x - b|} \le \Phi_1
			\]
			holds.
			\item\label{IntegralEstimatesii} There exists a constant $\Phi_2$ such that for all $\varepsilon > 0$ and for all $a,b \in \CC$ such that $|a|,|b| \le \varepsilon/2$ the inequality
			\[
			\int_{\CC \setminus \varepsilon\mathbb D} \left|\frac{1}{|x|^2} - \frac{1}{|x - a||x - b|}\right|idx \wedge d\bar{x} \le \Phi_2.
			\]
			\item\label{IntegralEstimatesiii} Let $r > 0$ and let $f$ be holomorphic in a neighborhood of $r\mathbb D$. Then there exists a constant $\Phi_3$ (that depends on $r$ and $f$) such that 
			\[
				\left|\int_{r\mathbb D \setminus \varepsilon \mathbb D} \dfrac{i|f(x)|^2dx \wedge d\bar{x}}{|x^2|^2} + 4\pi |f(0)|^2\ln(\varepsilon)\right| \le \Phi_3
			\]
			for all $\varepsilon \in (0, r)$.
		\end{enumerate}
	\end{lemma}
	\begin{proof}
		We only give a sketch of the proof, because the calculations are elementary. Statements~\ref{IntegralEstimatesi} and~\ref{IntegralEstimatesii} are proved by scaling the variable of integration $x$ by $1/\varepsilon$, after which both statements do not depend on $\varepsilon$ and can be proved by elementary estimations. The statement~\ref{IntegralEstimatesiii} is proved by noting that the function $x \mapsto (|f(x)|^2 - |f(0)|^2)/|x|^2$ is Lebesgue integrable on $r\mathbb D$, so we can replace $|f(x)|^2$ in the numerator with $|f(0)|^2$. After that the integral can be explicitly computed in polar coordinates.
	\end{proof}

	\begin{proof}[Proof of Lemma~\ref{lemQuasiReducedIfSubordinate}]
		Clearly, we can assume without loss of generality that $D_j \subset \CC$, $j = 1,2,3$. Moreover, we can assume that $a_1$ and $a_2$ are the first and the second column of $\mathfrak W(f;\mathfrak D)$ respectively. From the Riemann's bilinear relations~\cite[Proposition~III.2.3]{FarkasKra1980} we can calculate that for all $\begin{pmatrix} \alpha & \beta \end{pmatrix}^T \in \RR$ the equality
		\[
		\begin{pmatrix}
			\alpha & \beta
		\end{pmatrix} \;(\mathrm{Im}\;\Omega) \begin{pmatrix}
			\alpha \\ \beta
		\end{pmatrix} = \frac{i}{2}\int_{\mathcal X_f}\omega \wedge \bar{\omega}
		\]
		holds, where holomorphic $1$-form $\omega$ is normalized by the conditions
		\[
		\int_{\gamma_j} \omega = \alpha_j,\;\; j = 1,2,
		\]
		where $\gamma_j \in H_1(\mathcal X_f, \ZZ)$ is the cycle that corresponds to the period $a_j$ for $j = 1,2$. By definition of the periods $a_1$ and $a_2$ we, therefore, obtain that
		\begin{equation}\label{eqRiemannBilinearFormula}
			\begin{pmatrix}
				\alpha & \beta
			\end{pmatrix} \;(\mathrm{Im}\;\Omega) \begin{pmatrix}
				\alpha \\ \beta
			\end{pmatrix} = i\int_{\CC}\frac{|w(x)|^2dx \wedge d\bar{x}}{|f(x)|},
		\end{equation}
		where $w \in \mathfrak P_1$ is chosen to satisfy
		\begin{equation}\label{eqNormalizationByVector}
			\int_{\gamma_j}\frac{w(x)dx}{u(x)} = \alpha_j,\;\;j =1,2.
		\end{equation}
		Here $u$ is a single-valued branch of the square root of $f$ defined in a neighbourhood of $\CC\PP(1)\setminus(D_1 \cup D_2 \cup D_3)$ and $\gamma_j$ is the circle enclosing $D_j$ (note that the change of sign of the branch $u$ results in the sign change of $w$, so the formula~\eqref{eqRiemannBilinearFormula} is not affected).
		
		Now consider the sequence $f^{(n)}$ generated by the rules~\eqref{eqDefRichSeq}. Then from~\eqref{eqRiemannBilinearFormula} and Proposition~\ref{propPeriodsRelationWithSubordination} it follows that
		\begin{equation}\label{eqRiemannBilinearFormulaIteration}
			2^n\begin{pmatrix}
				\alpha & \beta
			\end{pmatrix} \;(\mathrm{Im}\;\Omega) \begin{pmatrix}
				\alpha \\ \beta
			\end{pmatrix} = i\int_{\CC}\frac{|w(x)|^2dx \wedge d\bar{x}}{|f^{(n)}(x)|},
		\end{equation}
		since the period matrix is doubled with each iteration of the Richelot construction and the normalization condition of the $1$-form $w(x)dx/y$ is not affected by it in view of Proposition~\ref{propPeriodsRelationWithSubordination}. With this preparation we are ready to state the main idea of the proof. Let $k \in E$ (see Definition~\ref{defQuasiReduced}) and $m \in \ZZ^2$ such that $m - k \in 2\ZZ^2$ and $m \ne \pm k$. We need to prove that $k^T (\mathrm{Im}\;\Omega) k < m^T (\mathrm{Im}\;\Omega) m$. For this it suffices to prove that 
		\[
		2^n(m^T (\mathrm{Im}\;\Omega) m - k^T (\mathrm{Im}\;\Omega) k) \to +\infty, \text{ when } n \to +\infty.
		\]
		We prove this by applying~\ref{eqRiemannBilinearFormulaIteration} and analyzing the behavior when $n \to +\infty$.
		
		For convenience we shall call two sequences of complex numbers {\it{equivalent}} if their difference is bounded. Also let $g(x) = c(x-t_1)^2(x-t_2)^2(x-t_3)^2$ denote the limit of the sequence $f^{(n)}$, and let $\delta_j^{(n)}$ denote the distance between the two roots of $f^{(n)}$ that belong to $D_j$. Now we claim that given $\begin{pmatrix} \alpha & \beta \end{pmatrix}^T \in \RR$ and the corresponding polynomial $w$ normalized by~\eqref{eqNormalizationByVector} the sequences
		\begin{multline}\label{eqEquivalentSequences}
			\left\{2^n\begin{pmatrix}
				\alpha & \beta
			\end{pmatrix} \;(\mathrm{Im}\;\Omega) \begin{pmatrix}
				\alpha \\ \beta
			\end{pmatrix}\right\}_{n \in \NN} \text{ and }\\
			\left\{  -\frac{4\pi|w(t_1)|^2 \ln\left(\delta_1^{(n)}\right)}{|c||t_1 - t_2|^2 |t_1 - t_3|^2} - \frac{4\pi|w(t_2)|^2 \ln\left(\delta_2^{(n)}\right)}{|c||t_2 - t_1|^2 |t_2 - t_3|^2 } - \frac{4\pi|w(t_1)|^2 \ln\left(\delta_3^{(n)}\right)}{|c||t_3 - t_1|^2 |t_3 - t_2|^2}  \right\}_{n \in \NN}
		\end{multline}
		are equivalent. To prove this at first consider closed disks $F_1,F_2,F_3$ such that $F_j \subset D_j$, the center of $F_j$ is $t_j$, and the smallest distance between any two of these disks is bigger than diameter of any of them. Clearly, as we are interested only in behavior when $n \to +\infty$ we can assume that the roots of $f$ belong to $F_1 \cup F_2 \cup F_3$. With this notation it is clear that the sequence
		\[
		\left\{i\int_{\CC \setminus {F_1 \cup F_2 \cup F_3}}\frac{|w(x)|^2dx \wedge d\bar{x}}{|f^{(n)}(x)|}\right\}_{n \in \NN}
		\]
		is bounded. Therefore, the sequences
		\[
		\left\{i\int_{\CC}\frac{|w(x)|^2dx \wedge d\bar{x}}{|f^{(n)}(x)|}\right\}_{n \in \NN} \text{ and } \left\{i\int_{F_1 \cup F_2 \cup F_3}\frac{|w(x)|^2dx \wedge d\bar{x}}{|f^{(n)}(x)|}\right\}_{n \in \NN}
		\]
		are equivalent. Further, let $G^{(n)}_j$ denote the closed disk with center $t_j$ and radius $2\delta_j^{(n)}$. From Lemma~\ref{lemConvergenceIfSmall}~\ref{ConvergenceIfSmalli} it is easy to see that if $a$ is a root of $f^{(n)}$ that belongs to $F_j$, then $|a - t_j| \le \delta_j^{(n)}$, since $t_j$ has to belong to a disk, whose diameter is the segment between the pair of roots of $f^{(n)}$ that belong to $F_j$. Therefore, Lemma~\ref{lemIntegralEstimates}~\ref{IntegralEstimatesi} implies that the sequence
		\[
		\left\{i\int_{G_j^{(n)}}\frac{|w(x)|^2dx \wedge d\bar{x}}{|f^{(n)}(x)|}\right\}_{n \in \NN}
		\]
		is bounded for all $j = 1,2,3$. Thus, the sequences
		\[
		\left\{i\int_{F_j}\frac{|w(x)|^2dx \wedge d\bar{x}}{|f^{(n)}(x)|}\right\}_{n \in \NN} \text{ and } \left\{i\int_{F_j \setminus G_j^{(n)}}\frac{|w(x)|^2dx \wedge d\bar{x}}{|f^{(n)}(x)|}\right\}_{n \in \NN}
		\]
		are equivalent. Further, by Lemma~\ref{lemIntegralEstimates}~\ref{IntegralEstimatesii} we can replace the roots of $f^{(n)}$ that belong to $F_j$ with the center of $F_j$ (i.e. with $t_j$) preserving the equivalence class of the sequence. That is, if $f^{(n)}(x) = c^{(n)}p_1^{(n)}(x)p_2^{(n)}(x)p_3^{(n)}(x)$, where $p_j^{(n)} \in \mathfrak P_2$ is monic and its roots belong to $F_j$, then the sequences
		\begin{multline*}
			\left\{i\int_{F_j \setminus G_j^{(n)}}\frac{|w(x)|^2dx \wedge d\bar{x}}{|f^{(n)}(x)|}\right\}_{n \in \NN} \text{ and }\\ \left\{i\int_{F_j \setminus G_j^{(n)}}\frac{|w(x)|^2dx \wedge d\bar{x}}{|c^{(n)}||x - t_j|^2\displaystyle\prod_{k \in \{1,2,3\}\setminus \{j\}}|p_k^{(n)}(x)|}\right\}_{n \in \NN}
		\end{multline*}
		are equivalent. Now note that
		\[
		\frac{1}{|c_n|\displaystyle\prod_{k \in \{1,2,3\}\setminus \{j\}}|p_k^{(n)}(x)|} \to \frac{1}{c\displaystyle\prod_{k \in \{1,2,3\}\setminus \{j\}}|x - t_k|^2}
		\]
		for all $x \in F_j$ and the convergence is uniformly quadratic (i.e. the absolute value of the difference can be uniformly estimated from above by a sequence that converges to $0$ quadratically). On the other hand, in view of Lemma~\ref{lemIntegralEstimates}~\ref{IntegralEstimatesiii}, for all $n \in \NN$ we have
		\[
		i\int_{F_j \setminus G_j^{(n)}}\frac{|w(x)|^2dx \wedge d\bar{x}}{|x - t_j|^2} \le -A\ln\left(\delta_j^{(n)}\right) + B
		\]
		with suitable constants $A > 0$ and $B \in \RR$. Moreover, from Lemma~\ref{lemConvergenceIfSmall}~\ref{ConvergenceIfSmallii} we know that the convergence $\left(\delta_j^{(n)}\right) \to 0$ is strictly quadratic, i.e. growth of $-\ln\left(\delta_j^{(n)}\right)$ is estimated from above by $C2^n + D$ with a suitable constants $C > 0$ and $D \in \RR$. Since the product of a sequence that quadratically converges to $0$ by a geometric progression still converges to zero, we conclude that the sequences
		\begin{multline}\label{eqFinalSequence}
			\left\{i\int_{F_j \setminus G_j^{(n)}}\frac{|w(x)|^2dx \wedge d\bar{x}}{|c^{(n)}||x - t_j|^2\displaystyle\prod_{k \in \{1,2,3\}\setminus \{j\}}|p_k^{(n)}(x)|}\right\}_{n \in \NN} \text{ and }\\ \left\{i\int_{F_j \setminus G_j^{(n)}}\frac{|w(x)|^2dx \wedge d\bar{x}}{|c||x - t_1|^2|x - t_2|^2|x - t_3|^2}\right\}_{n \in \NN}
		\end{multline}
		are equivalent. The asymptotic behavior of the second sequence from~\eqref{eqFinalSequence} is given by Lemma~\ref{lemIntegralEstimates}~\ref{IntegralEstimatesiii}: it is equivalent to the sequence of corresponding summands in~\eqref{eqEquivalentSequences}. Thus, through a chain of equivalences we showed the equivalence of the sequences in~\eqref{eqEquivalentSequences}. Now we make a step further and calculate $w(t_j)$ for $j = 1,2,3$. To do this note that $w$ is normalized to satisfy~\eqref{eqNormalizationByVector} and, since this normalization does not depend on $n$, we obtain
		\[
		\int_{\gamma_j}\frac{w(x)dx}{\sqrt{c}(x - t_1)(x - t_2)(x - t_3)} = \alpha_j
		\]
		by taking limit when $n \to +\infty$ (note that this holds for $j = 3$ if we put $\alpha_3 = -\alpha_1 - \alpha_2$). Thus, by calculating integrals using residues it is easy to obtain that 
		\[
		\begin{gathered}
			w(t_1) = \frac{1}{2\pi i} \alpha_1\sqrt{c}(t_1 - t_2)(t_1 - t_3),\;\; w(t_2) = \frac{1}{2\pi i} \alpha_2\sqrt{c}(t_2 - t_1)(t_2 - t_3),\\
			w(t_3) = -\frac{1}{2\pi i} (\alpha_1 + \alpha_2)\sqrt{c}(t_3 - t_1)(t_3 - t_2).
		\end{gathered}
		\]
		By substituting the result into~\eqref{eqEquivalentSequences} we get that the sequences
		\begin{multline}\label{eqEquivalentSequencesFinal}
			\left\{2^n\begin{pmatrix}
				\alpha & \beta
			\end{pmatrix} \;(\mathrm{Im}\;\Omega) \begin{pmatrix}
				\alpha \\ \beta
			\end{pmatrix}\right\}_{n \in \NN} \text{ and }\\
			\left\{ -\frac{\alpha_1^2 \ln\left(\delta_1^{(n)}\right) + \alpha_2^2 \ln\left(\delta_2^{(n)}\right) +(\alpha_1 + \alpha_2)^2 \ln\left(\delta_1^{(n)}\right)}{\pi}   \right\}_{n \in \NN}
		\end{multline}
		are equivalent.
		
		With equivalence of the sequences~\eqref{eqEquivalentSequencesFinal} we can finish the proof of Lemma~\ref{lemQuasiReducedIfSubordinate}. Indeed, consider for example $k = \begin{pmatrix} 1 & 0 \end{pmatrix}^T$ and any $m = \begin{pmatrix} m_1 & m_2\end{pmatrix}^T$ such that $m - k \in 2\ZZ^2$. From the equivalence of sequences~\eqref{eqEquivalentSequencesFinal} we get that the sequences
		\begin{multline}\label{eqEquivalentSequencesExample}
			\left\{2^n(m^T (\mathrm{Im}\;\Omega) m - k^T (\mathrm{Im}\;\Omega) k) \right\}_{n \in \NN} \text{ and } \\ \left\{ -\frac{(m_1^2 - 1) \ln\left(\delta_1^{(n)}\right) + m_2^2 \ln\left(\delta_2^{(n)}\right) +((m_1 + m_2)^2-1) \ln\left(\delta_1^{(n)}\right)}{\pi}   \right\}_{n \in \NN}
		\end{multline}
		are equivalent. Since $m - k \in 2\ZZ^2$ we conclude that $m_1$ and $m_1 + m_2$ are odd integer numbers. Therefore, the second sequence in~\eqref{eqEquivalentSequencesExample} converges to $+\infty$, unless $m_1^2 = 1$, $m_2 = 0$, and $(m_1 + m_2)^2 = 1$. Clearly, this is possible only if $m = \pm k$. Similar argument works for all vectors $k \in E$, and the proof is finished.
	\end{proof}
	\begin{proof}[Proof of Theorem~\ref{thPassingToTheLimit}]
		We give a proof only for the sequence $S^{f^{(n)}}$, as the proof is the same for other Kleinian functions (we only need the fact of convergence of Taylor coefficients of order $\le 2$, which holds for $S^{f^{(n)}}$ by definition, and for functions $S_{jk}^{f^{(n)}}$ by~\eqref{eqTaylorExpansions}). Let $a_1,a_2,b_1,b_2$ be a symplectic basis in $\Per_f$ such that $a_1$ and $a_2$ are the first two columns of $\mathfrak W(f;\mathfrak D)$. Let $A$ and $B$ be formed of the columns $a_1,a_2$ and $b_1,b_2$ respectively and put $\Omega = A^{-1}B$. By Lemma~\ref{lemQuasiReducedIfSubordinate} $\Omega$ is quasi-reduced. Moreover, since $a_1,a_2 \in \Per_{f^{(n)}}$ we can consider matrices $\eta_A^{(n)}$ formed of the columns $\eta^{f^{(n)}}(a_1)$, $\eta^{f^{(n)}}(a_2)$. From the definition it is clear that matrices $\eta_A^{(n)}$ converge to the matrix $\eta^{(\infty)}_A$ that consists of the first two rows of $\mathfrak E(g;\mathfrak D)$. Now consider the isomorphisms $T^{(n)}:R_2^{2^n\Omega} \to \mathfrak S_{f^{(n)}}$ defined in~\eqref{eqTisomorphism} and let $\phi^{(n)} = \left(T^{(n)}\right)^{-1}\left(S^{f^{(n)}}\right) \in R_2^{2^n\Omega}$. It is easy to see that for $S \in \mathfrak S_{f^{(n)}}$ we have
		\[
		\left(T^{(n)}\right)^{-1}\left(S\right)(z) = \exp\left[-z^TA^T\eta^{(n)}_Az\right]S(Az).
		\]
		From this formula it is evident that the sequence of order $2$ Taylor expansions of functions $\phi^{(n)}$ converges, thus by Lemma~\ref{lemThetaConvergence} the sequence $\phi^{(n)}$ converges uniformly on compact sets to a function that belongs to $\Span\{1, \sin^2(\pi z_1), \sin^2(\pi z_2), \sin^2(\pi (z_1 + z_2))\}$. From this it follows that the sequence $S^{f^{(n)}}$ uniformly on compact sets converges to a function that has the form
		\[
		\exp[z^T\eta_A^{(\infty)}A^{-1}z](\alpha + \beta \sin^2(\pi r_1z) + \gamma \sin^2(\pi r_2 z) + \delta \sin^2(\pi r_3z)),
		\]
		where $r_1$ and $r_2$ are the rows of $A^{-1}$ and $r_3 = r_1 - r_2$. It remains to note that by definition the matrix $L$ has the rows $\pm r_1,\pm r_2,\pm r_3$, and $M = \eta_A^{(\infty)}A^{-1}$.
	\end{proof}
	
	\section{Algorithms and numerical experiments}\label{sec:Algo}
	
	\subsection{Algorithm to compute periods.} Here we present a variation of a well-known algorithm (see, e.g.~\cite{BostMestre} and~\cite{Chow}) that computes certain periods of the curve $\mathcal X_f$. We formulate the algorithm for admissible polynomials $f$ that are subordinate to a triple of disks. The main novelty that we introduce is an additional step in the algorithm that allows to compute the values $\eta^f(w)$ for calculated periods $w$. The idea is to use the formula~\eqref{eqEMatrixTransform} and the explicit expression~\eqref{eqChiFormula} for the matrix $\mathfrak H^{p,q,r}$. 
	
	Consider a polynomial $f \in \mathfrak P_6$ that is subordinate to $\mathfrak D = (D_1,D_2,D_3)$. Below we present an algorithm to compute the matrices $\mathfrak W(f;\mathfrak D)$ and $\mathfrak E(f;\mathfrak D)$ defined in~\eqref{eqWMatrix} and~\eqref{eqEMatrix} respectively.
	
	\begin{algo}\label{algPeriods}
		\begin{enumerate}
			\item Calculate the sequence of polynomials $f^{(n)}$, $n = 0, \dots, N$, where $f^{(0)} = f$ and $f^{(n+1)} = H_{\mathfrak D}(f^{(n)})$ for $n = 1, \dots, N-1$. $N$ should be chosen large enough, so that the polynomial $f^{(N)}$ is sufficiently close to a polynomial $g(x) = C(x - t_1)^2(x - t_2)^2(x - t_3)^2$, where $t_j \in D_j$ for $j = 1,2,3$.
			\item Compute the matrices 
			\[ \mathfrak W(g; \mathfrak D),\;\; \mathfrak E(g;\mathfrak D)\]
			by using analytic formulas (which are trivial to obtain via calculus of residues). The matrix $\mathfrak W(g; \mathfrak D)$ is already an approximation of the matrix $\mathfrak W(f;\mathfrak D)$.
			\item Calculate recursively the matrices $\mathfrak E^{(n)}$ by the rules $\mathfrak E^{(N)} = \mathfrak E(g;\mathfrak D)$ and
			\[
			\mathfrak E^{(n)} = 2\mathfrak E^{(n+1)} + \mathfrak H^{p_1^{(n)}p_2^{(n)}p_3^{(n)}} \mathfrak W(g;\mathfrak D) \text{ for } n = N-1,\dots,0,
			\]
			where $f^{(n)} = p_1^{(n)}p_2^{(n)}p_3^{(n)}$ is the decomposition associated with the triple $\mathfrak D$. The matrix $\mathfrak E^{(0)}$ is an approximation of $\mathfrak E(f;\mathfrak D)$.
		\end{enumerate}    
	\end{algo}
	
	\begin{remark}
		It is noteworthy that if all the roots of $f$ lie on a common circle then it is possible to formulate an algorithm that does not require a triple of disks and computes a full symplectic basis in $\Per_f$ rather than a couple of vectors that generate a maximal isotropic subgroup in $\Per_f$. The idea is to arrange the roots of the polynomial on each step into three pairs according to the order of the roots on the circle, i.e. each pair consists of neighbouring roots (it is easy to see that this arrangement is in fact subordinate to a suitable triple of disks, so the algorithm can be viewed as a special case as Algorithm~\ref{algPeriods}). There are essentially two ways of partitioning six points on a circle into three pairs according to their order and by running the algorithm for both cases yields a symplectic basis in $\Per_f$. This version of the algorithm is presented in~\cite{BostMestre}. Using the idea of the third step of Algorithm~\ref{algPeriods} it can be modified to compute values $\eta^f$ at the basis vectors as well.
	\end{remark}
	
	\subsection{Algorithm to compute values of Kleinian functions}
	Now we formulate the algorithms to compute canonical Kleinian functions of weight $2$ (and their first derivatives) associated with a polynomial $f$ that is subordinate to a triple of disks $\mathfrak D = (D_1,D_2,D_3)$. For completeness we provide the additional steps for this algorithm that compute other special functions, namely, $\wp^f_{jk}$ for general $f$, and classical Kleinian functions $\sigma^f$ and $\zeta_j^f$ if $f$ is in Weierstrass form, i.e. $\deg f = 5$ and $f_5 = 4$. The relations involving these functions and Kleinian functions of weight 2, that are necessary for the algorithm below are presented in~\cite{KleinianWeight2}.
	
	As usual, we assume that $f \in \mathfrak P_6$ is an admissible polynomial that is subordinate to $\mathfrak D = (D_1,D_2,D_3)$. Moreover, fix $z \in \CC^2$. The following algorithm computes approximations of Kleinian functions associated with the polynomial $f$ at the point $z$. We labelled by an asterisk the last step of it to emphasize that it is applicable only to polynomials in Weierstrass form.
	\begin{algo}\label{algFunctions}
		\begin{enumerate}
			\item Calculate the sequence of polynomials $f^{(n)}$, $n = 0, \dots, N$, where $f^{(0)} = f$ and $f^{(n+1)} = H_{\mathfrak D}(f^{(n)})$ for $n = 1, \dots, N-1$. $N$ should be chosen large enough, so that the polynomial $f^{(N)}$ is sufficiently close to a polynomial $g(x) = C(x - t_1)^2(x - t_2)^2(x - t_3)^2$, where $t_j \in D_j$ for $j = 1,2,3$.
			\item For each $n = 0,\dots,N-1$ calculate the matrices $\mathfrak H^{(n)} = \mathfrak H^{p_1^{(n)}, p_2^{(n)}, p_3^{(n)}}$, $\mathfrak A_{p_1^{(n)}, p_2^{(n)}, p_3^{(n)}}$, and $\mathfrak A_{p_1^{(n)}, p_2^{(n)}, p_3^{(n)}}^{(jk)}$, where $f^{(n)} = p_1^{(n)}p_2^{(n)}p_3^{(n)}$ is the decomposition associated with the triple $\mathfrak D$.
			\item Using the formulas~\eqref{eqSAllLimit} (and their derivatives) calculate approximations $\mathscr S_{N}$, $\mathscr S_{N}^{(1)}$, $\mathscr S_{N}^{(2)}$ of the vectors
			\[
			\mathscr S_{f^{(N)}}(z), \;\;\frac{\partial \mathscr S_{f^{(N)}}}{\partial z_1}(z), \;\; \frac{\partial \mathscr S_{f^{(N)}}}{\partial z_2}(z)
			\]
			respectively.
			\item For all $n = N - 1, \dots, 0$ recursively calculate approximations $\mathscr S_{n}$, $\mathscr S_{n}^{(1)}$, $\mathscr S_{n}^{(2)}$ of the vectors
			\[
			\mathscr S_{f^{(n)}}(z), \;\;\frac{\partial \mathscr S_{f^{(n)}}}{\partial z_1}(z), \;\; \frac{\partial \mathscr S_{f^{(n)}}}{\partial z_2}(z)
			\]
			by substituting $\mathscr S_{n+1}$, $\mathscr S_{n+1}^{(1)}$, $\mathscr S_{n+1}^{(2)}$ into the right-hand side of the formula~\eqref{eqSMainTransform} (and formulas that are obtained from it by differentiating).
			\item The vectors
			\[
			\mathscr S_{0} = \begin{pmatrix}
				S & S_{22} & S_{12} & S_{11}
			\end{pmatrix}^T,\;\;\mathscr S_{0}^{(j)} = \begin{pmatrix}
				S^{(j)} & S_{22}^{(j)} & S_{12}^{(j)} & S_{11}^{(j)}
			\end{pmatrix}^T,\; j =1,2
			\]
			are approximations of the vectors
			\[
			\mathscr S_{f}(z), \;\;\frac{\partial \mathscr S_{f}}{\partial z_j}(z),\;j = 1,2.
			\]
			Assuming that $S^f(z) \ne 0$ (i.e. $z$ does not belong to the polar set of functions $\wp_{jk}^f$) the values $S_{jk}/S$ are the approximations of the values $\wp_{jk}^f$ for all $(j,k) = (1,1),(1,2),(2,2)$. \stepcounter{enumi}
			\item[(\arabic{enumi}*)] If $f$ is in Weierstrass form, then approximate $\sigma^f(2z)$ using duplication formula~\cite[Corollary~3.9]{KleinianWeight2}. That is, the value
			\[
			S_{12}S_{22}^{(1)} - S_{22}S_{12}^{(1)} + S_{11}S^{(1)} - SS_{11}^{(1)}
			\]
			is an approximation of $\sigma^f(2z)$. If $S^f(z) \ne 0$, then the values 
			\[
			\frac{S^{(1)}}{2S},\;\;\frac{S^{(2)}}{2S}
			\]
			are the approximations of $\zeta_1(z)$ and $\zeta_2(z)$.
		\end{enumerate}
	\end{algo}
	\begin{remarks} \hfill
		\begin{enumerate}[label=\arabic*.]
			\item Note that in order to calculate $\sigma^f(z)$ Algorithm~\ref{algFunctions} should be performed at the point $z/2$.
			\item Since the first two steps of the Algorithm~\ref{algFunctions} do not depend on $z$, the output of these steps can be precomputed and stored, if multiple calculations with a single polynomial $f$ are required.
			\item For the computation of the values $\mathscr S_f(z)$ and $\wp_{jk}^f(z)$ it suffices to perform a simpler version of Algorithm~\ref{algFunctions} which does not compute the derivatives of $\mathscr S_f(z)$.
			\item The main reason why the emphasis of Algorithm~\ref{algFunctions} is placed on the computation of canonical Kleinian functions of weight $2$ and their first derivatives, is the entire analyticity of the vector-valued function $\mathscr S_f$. It is easy to reformulate our results to obtain similar to Algorithm~\ref{algFunctions} computational procedures in terms of meromomorphic functions $\wp_{jk}^f$ and/or their derivatives. To our experience the performance of such algorithms in close vicinity of the polar locus drops significantly in comparison with Algorithm~\ref{algFunctions}. On the other hand, canonical Kleinian functions of weight $2$ and their first derivatives are sufficient to calculate the values of all other special functions that we consider in this work and their derivatives (for higher derivatives the expressions can be obtained via~\cite[Proposition~4.6]{KleinianWeight2}).
		\end{enumerate}
	\end{remarks}
	
	\subsection{Algorithm to compute the Abel map}
	For completeness we also sketch an algorithm that computes the Abel map $\mathscr A_f: \mathcal X_f^{(2)} \to \Jac_f$ (this map is defined on a degree $2$ divisor $D$ as the Abel map of $D - \mathfrak L_f$, where $\mathfrak L_f$ is the canonical hyperelliptic divisor class on $\mathcal X_f$; see~\cite{KleinianWeight2}). We shall exploit the fact, that the meromorphic functions $D \mapsto \wp_{jk}^f(\mathscr A_f(D))$ on $\mathcal X_f^{(2)}$ can be explicitly computed in terms of $D$. For these functions (which are denoted as $\xi_{jk}^f$) were given explicit formulas in~\cite[Section~2.2]{KleinianWeight2}). 
	
	Now we discuss the idea of the algorithm. Assume we are given $D \in \mathcal X_f^{(2)}$. Since we can compute $\wp$-functions at $z = \mathscr A_f(D)$, we can find the vector $\mathscr S_f(z)$ modulo multiplication by a scalar (that is, the vector with homogeneous coordinates $(1:\wp_{22}^f(z), \wp_{12}^f(z), \wp_{11}^f(z))$). That is, we can find the point on the Kummer surface of $\mathcal X_f$, which corresponds to $z$. Now, using the explicit formulas of~\cite{RichelotExpression}, it is possible to find the preimages of a given point with respect to the Richelot isogeny. Inverting the Richelot isogeny several times we arrive at the problem of calculating the Abel map for a curve $\mathcal X_g$ with $g$ being sufficiently close to a square of a polynomial of degree $3$, for which the Abel map can be easily approximated. However, this computation will allow us only to compute the vector $z$ up to a sign change. To overcome this indeterminacy it is possible to apply Algorithm~\ref{algFunctions} and compute values of several odd meromorphic functions at $z$ and compare the result with explicit formulas that use coordinates of summands in the divisor $D$. 
	
	As usual, let $f \in \mathfrak P_6$ be an admissible polynomial that is subordinate to $\mathfrak D = (D_1,D_2,D_3)$. Moreover, fix $D = (P) + (Q) \in \mathcal X_f^{(2)}$. The following algorithm calculates a vector that is approximately $\mathscr A_f(D)$ modulo $\Per_f$. The final step of this algorithm is divided into two possibilities. That is, we separately treat the divisors $D$ such that one the summands (either $P$, or $Q$) is at infinity.
	
	\begin{algo}\label{algAbelMap}
		\begin{enumerate}
			\item Calculate the sequence of polynomials $f^{(n)}$, $n = 0, \dots, N$, where $f^{(0)} = f$ and $f^{(n+1)} = H_{\mathfrak D}(f^{(n)})$ for $n = 1, \dots, N-1$. $N$ should be chosen large enough, so that the polynomial $f^{(N)}$ is sufficiently close to a polynomial $g(x) = c(x - t_1)^2(x - t_2)^2(x - t_3)^2$, where $t_j \in D_j$ for $j = 1,2,3$.
			\item Calculate the vector $v_0$ that consists of homogeneous coordinates of $\mathscr S_f(\mathscr A_f(D))$ using explicit expressions for the values $\xi_{jk}^f(D)$ (see~\cite[Section~2.2]{KleinianWeight2}). Normalize $v_0$ so that its Euclidean norm is $1$.
			\item\label{algAbelMapStepPreimage} For each $n = 1,2,\dots N$ calculate the vector $v_n$ such that the following properties hold.
			\begin{enumerate}[label=(\alph*)]
				\item\label{preima} The Euclidean norm of $v_n$ is equal to $1$.
				\item\label{preimb} The element $u_n$ of the space $\CC\PP(3)$, whose homogeneous coordinates are components of $v_n$, belongs to $\mathcal K_{f^{(n)}}$ (the Kummer surface of $\mathcal X_f$; see~\cite[Section~2.3]{RichelotExpression}). Moreover, the Richelot isogeny $\mathcal K_{f^{(n)}} \to \mathcal K_{f^{(n-1)}}$ maps $u_n$ to $u_{n-1}$.
				\item\label{preimc} The vector $v_n$ is the closest vector to $v_{n-1}$ among those vectors that satisfy~\ref{preima} and~\ref{preimb}.
			\end{enumerate}
			\item Let $v_N = \begin{pmatrix}
				\alpha & \beta & \gamma & \delta
			\end{pmatrix}^T$. Then calculate the roots $x_1$ and $x_2$ of the polynomial $\alpha x^2 - \beta x - \gamma$. 
			\item Let
			\[z = \begin{pmatrix}
				\displaystyle\int_{x_1}^{x_2} \frac{dx}{\sqrt{c}(x - t_1)(x- t_2)(x - t_3)} & \displaystyle\int_{x_1}^{x_2} \frac{xdx}{\sqrt{c}(x - t_1)(x- t_2)(x - t_3)}
			\end{pmatrix}^T.\]
			The vector $z$ is an approximation either for $\mathscr A_f(D)$, or $-\mathscr A_f(D)$.  \stepcounter{enumi}
			\item[{\crtcrossreflabel{(\arabic{enumi}a)}[algStepSignFinite]}] Assume that $D = (P) + (Q)$, where both $P = (x_1,y_1)$ and $Q = (x_2,y_2)$ are not at infinity. Then use Algorithm~\ref{algFunctions} to calculate
			\[
			\frac{\partial \wp_{22}^f}{\partial z_2}(z),\;\;\frac{\partial \wp_{12}^f}{\partial z_2}(z).
			\]
			On the other hand calculate
			\[
			\frac{\partial \wp_{22}^f}{\partial z_2}(\mathscr A_f(D)) = \frac{y_2 - y_1}{x_2 - x_1},\;\;\frac{\partial \wp_{12}^f}{\partial z_2}(\mathscr A_f(D)) = \frac{x_2 y_1 - x_1 y_2}{x_1 - x_2}.
			\]
			Compare these values to decide which of the vectors $\pm z$ is an approximation of $\mathscr A_f(D)$.
			\item[{\crtcrossreflabel{(\arabic{enumi}b)}[algStepSignInfinite]}] Assume that $D = (P) + (Q)$, where $P = (x,y)$, $x \ne 0$, and $Q$ is at infinity. Let $a$ be the value of the meromorphic function $y/x^3$ at $Q$ (so $a^2 = f_6$). Use Algorithm~\ref{algFunctions} to calculate
			\[
			\frac{\partial \ln S_{22}^f}{\partial z_2}(z) - \frac{\partial \ln S_{12}^f}{\partial z_2}(z),\;\;\frac{\partial \ln S_{22}^f}{\partial z_1}(z) - \frac{\partial \ln S_{12}^f}{\partial z_1}(z).
			\]
			On the other hand, compute
			\[
			\begin{gathered}
				\frac{\partial \ln S_{22}^f}{\partial z_2}(\mathscr A_f(D)) - \frac{\partial \ln S_{12}^f}{\partial z_2}(\mathscr A_f(D)) = -ax,\\ \frac{\partial \ln S_{22}^f}{\partial z_1}(\mathscr A_f(D)) - \frac{\partial \ln S_{12}^f}{\partial z_1}(\mathscr A_f(D)) = ax^2 - \frac{y}{x}.
			\end{gathered}
			\]
			Again, decide on the sign of the vector $z$ by comparing these calculations.
		\end{enumerate}
	\end{algo}
	\begin{remarks}\hfill
		\begin{enumerate}
			\item The most non-trivial part of the Algorithm~\ref{algAbelMap} is Step~\ref{algAbelMapStepPreimage}. There are several ways to find all preimages of a given point in the Kummer surface with respect to Richelot isogeny. The simplest of them is to use the expression for Richelot isogeny given in \cite[Theorem~4.4]{RichelotExpression}. Using the formula given there it is easy to find that there are $8$ distinct preimages modulo multiplication by scalars. Only $4$ of them belong to the Kummer surface and to determine those one can substitute their coordinates into the equation of the Kummer surface (which is the relation of \cite[Proposition~2.2]{KleinianWeight2}).
			\item The explicit formulas for the derivatives of meromorphic functions on $\Jac_f$ used in Steps~\ref{algStepSignFinite} and~\ref{algStepSignInfinite} were obtained by the method described in the proof of \cite[Proposition~4.6]{KleinianWeight2}. Moreover, it can be verified that for a divisor $D = (P) + (Q)$, where both $P = (x_1,y_1)$ and $Q = (x_2,y_2)$ are not at infinity, the equalities
			\[
			\begin{gathered}
				\frac{\partial \ln S_{22}^f}{\partial z_2}(\mathscr A_f(D)) - \frac{\partial \ln S_{12}^f}{\partial z_2}(\mathscr A_f(D)) = \frac{x_2^2 y_1 - x_1^2 y_2}{x_1x_2(x_2^2 - x_1^2)},\\ \frac{\partial \ln S_{22}^f}{\partial z_1}(\mathscr A_f(D)) - \frac{\partial \ln S_{12}^f}{\partial z_1}(\mathscr A_f(D)) = \frac{x_1^3 y_2 - x_2^3 y_1}{x_1x_2(x_2^2 - x_1^2)}
			\end{gathered}
			\]
			hold. Thus, the functions that we used to test the sign in Step~\ref{algStepSignInfinite} are applicable to divisors such that $x$-coordinates of both summands do not vanish. 
			\item The Step~\ref{algStepSignFinite} has several subtleties that can occur in special cases. At first note that if $P = Q$, then the calculation of the values 
			\[
			\frac{\partial \wp_{22}^f}{\partial z_2}(\mathscr A_f(D)) = \frac{y_2 - y_1}{x_2 - x_1},\;\;\frac{\partial \wp_{12}^f}{\partial z_2}(z) = \frac{x_2 y_1 - x_1 y_2}{x_1 - x_2}
			\] 
			requires passing to the limit. Moreover, it is possible that both these values are equal to $0$, so these values do not help to determine the sign of $z$. Nevertheless, in this case it is easy to see that both $P$ and $Q$ have to be Weierstrass points, so $z$ is a point of order $2$. That is, the sign does not matter in this case.
			\item In order to simplify dealing with special cases of the Steps~\ref{algStepSignFinite} and~\ref{algStepSignInfinite} it is possible to implement the following procedure. Generate a random point $R \in \mathcal X_f$ and apply Algorithm~\ref{algAbelMap} to divisors $D' = (P) + (R)$ and $D'' = (\mathscr J_f(R)) + (Q)$. Since the point $R$ is chosen randomly, we can assume that $D'$ and $D''$ are generic. Finally, we can calculate $\mathscr A_f(D) = \mathscr A_f(D') + \mathscr A_f(D'')$.
		\end{enumerate}
	\end{remarks}
	
	\printbibliography
	
\end{document}